\documentclass[11pt,leqno]{amsart}

\usepackage{amsthm,amsfonts,amssymb,amsmath,oldgerm}
\numberwithin{equation}{section}
\usepackage{fullpage}
\usepackage{amsmath}
\usepackage{amsfonts}
\usepackage{amssymb}
\usepackage{setspace}
\usepackage{stmaryrd}
\usepackage{mathrsfs}
\usepackage{fancyhdr}
\usepackage{graphicx}
\usepackage{enumerate}
\usepackage{bm}
\usepackage{multirow}
\usepackage{psfrag}
\usepackage[font=scriptsize]{caption}
\usepackage[font=scriptsize]{subcaption}
\usepackage{listings}
\usepackage{tikz}
\usetikzlibrary{matrix} 
\usepackage[framemethod=TikZ]{mdframed}
\mdfdefinestyle{MyFrame}{backgroundcolor=gray!20!white}
\usepackage[makeroom]{cancel}

\usepackage[top=30mm,bottom=30mm,left=22mm,right=22mm,a4paper]{geometry}

\usepackage{hyperref,bookmark}

\usepackage[normalem]{ulem}
\definecolor{violet}{rgb}{0.580,0.,0.827}

\usepackage[textwidth= \marginparwidth,textsize=tiny]{todonotes}

\hypersetup{
    colorlinks,          
    citecolor=blue,        
    filecolor=blue,      
    urlcolor=blue,           
    linkcolor=blue,
}

\renewcommand\d{\partial}
\newcommand\dD{\mathrm{d}}

\newcommand\dd{\dD}
\def\eps{\varepsilon }

\newcommand{\Id}{{\rm Id}}
\newcommand{\A}{{\rm \bA}}
\newcommand{\J}{{\rm \bJ}}

\newcommand\br{\begin{remark}}
\newcommand\er{\end{remark}}
\newcommand\bp{\begin{pmatrix}}
\newcommand\ep{\end{pmatrix}}
\newcommand{\be}{\begin{equation}}
\newcommand{\ee}{\end{equation}}
\newcommand\ba{\begin{equation}\begin{aligned}}
\newcommand\ea{\end{aligned}\end{equation}}
\newcommand\ds{\displaystyle}

\newcommand{\beg}{\begin{example}}
\newcommand{\eeg}{\end{exaplem}}
\newcommand{\bpr}{\begin{proposition}}
\newcommand{\epr}{\end{proposition}}
\newcommand{\bt}{\begin{theorem}}
\newcommand{\et}{\end{theorem}}
\newcommand{\bc}{\begin{corollary}}
\newcommand{\ec}{\end{corollary}}
\newcommand{\bl}{\begin{lemma}}
\newcommand{\el}{\end{lemma}}
\newcommand{\bd}{\begin{definition}}
\newcommand{\ed}{\end{definition}}
\newcommand{\brs}{\begin{remarks}}
\newcommand{\ers}{\end{remarks}}

\newtheorem{theorem}{Theorem}[section]
\newtheorem{proposition}[theorem]{Proposition}
\newtheorem{corollary}[theorem]{Corollary}
\newtheorem{lemma}[theorem]{Lemma}
\newtheorem{remark}[theorem]{Remark}
\newtheorem{definition}[theorem]{Definition}

\newtheorem{example}[theorem]{Example}
\newtheorem{notation}[theorem]{Notational convention}

\newcommand{\eg}{\textit{e.g.}}
\newcommand{\ie}{\textit{i.e.}}
\newcommand{\cf}{\textit{cf.} }

\def\mm#1{\begin{align}#1\end{align}}
\def\mmn#1{\begin{align*}#1\end{align*}}

\newcommand{\N}{\mathbf N}

\DeclareMathOperator{\Div}{div}

\newcommand{\RR}{{\mathbb R}}

\newcommand\bx{{\bm x}}
\newcommand{\bz}{\bm z}
\newcommand{\by}{\bm y}

\newcommand{\ber}{\bm e}

\newcommand\bv{{\bm v}}

\newcommand\bA{{\mathbf A}}
\newcommand\bB{{\mathbf B}}

\newcommand\bE{{\mathbf E}}
\newcommand\bF{{\mathbf F}}
\newcommand\bG{{\mathbf G}}

\newcommand\bJ{{\mathbf J}}

\newcommand\bV{{\mathbf V}}

\newcommand\bX{{\mathbf X}}
\newcommand\bY{{\mathbf Y}}

\newcommand\bfa{{\bm a}}
\newcommand\bfb{{\bm b}}

\newcommand\bfv{{\bm v}}
\newcommand\bfw{{\bm w}}
\newcommand\bfx{{\bm x}}
\newcommand\bfy{{\bm y}}
\newcommand\bfz{{\bm z}}

\newcommand\bftau{{\bm \tau}}

\newcommand\cA{{\mathcal A}}

\newcommand\cO{{\mathcal O}}

\newcommand\cX{{\mathcal X}}

\newcommand\tbv{\widetilde{\bfv}}
\newcommand\tbw{\widetilde{\bfw}}

\newcommand\tbz{\widetilde{\bfz}}

\newcommand{\ttime}{\widetilde{t}}
\newcommand{\ttauv}{\widetilde{\bftau}_{\bv}}
\newcommand{\htime}{\widehat{t}}
\newcommand{\hx}{\widehat{\bfx}}
\newcommand{\hy}{\widehat{\bfy}}
\newcommand{\hv}{\widehat{\bfv}}
\newcommand{\hw}{\widehat{\bfw}}
\newcommand{\htauy}{\widehat{\bftau}_{\bfy}}
\newcommand{\htauv}{\widehat{\bftau}_{\bfv}}

\newcommand{\Dt}{\Delta t}

\title{
Convergence analysis of asymptotic preserving schemes for strongly magnetized plasmas
}

\author{Francis Filbet}
\address{Universit\'e de Toulouse III \& IUF, 
UMR CNRS 5219, Institut de Math\'ematiques de Toulouse,
118 route de Narbonne,
F-31062 Toulouse cedex, France}
\email{francis.filbet@math.univ-toulouse.fr }
\thanks{FF is supported by the EUROfusion Consortium and has received funding
from the Euratom research and training programme 2014--2018 under the grant
agreement No 633053. The views and opinions expressed herein do not
necessarily reflect those of the European Commission.\\[-0.5em]}

\author{L.~Miguel Rodrigues}
\address{
Universit\'e de Rennes 1 \& IUF, CNRS, IRMAR - UMR 6625, F-35000 Rennes, France}
\email{{\tt luis-miguel.rodrigues@univ-rennes1.fr}}
\thanks{Research of LMR has received funding from the city of Rennes.\\[-0.5em]
}

\author{Hamed Zakerzadeh}
\address{Universit\'e de Toulouse III, UMR CNRS 5219, Institut de Math\'{e}matiques de Toulouse,
118 route de Narbonne,
F-31062 Toulouse cedex, France}
\email{seyed\_hamed.zakerzadeh@math.univ-toulouse.fr}
\thanks{HZ was supported by LabEx CIMI Toulouse and, partially, by Institut Universitaire de France.}

\begin{document}

\begin{abstract}
The present paper is devoted to the convergence analysis of a class of asymptotic preserving particle schemes \textit{[Filbet \& Rodrigues, \textit{SIAM J.\ Numer.\ Anal.}, 54(2) (2016):1120--1146]} for the Vlasov equation with a strong external magnetic field. In this regime, classical Particle-In-Cell (PIC) methods are subject to quite restrictive stability constraints on the time and space steps, {due to} the small Larmor radius and plasma frequency. The asymptotic preserving discretization that we are going to study removes such a constraint while capturing the large-scale dynamics, even when the discretization (in time and space) is too coarse to capture fastest scales. Our error bounds are explicit regarding the discretization, stiffness parameter, initial data and time.
\end{abstract}

\date{\today}
\maketitle

\vspace{0.1cm}
\\

\vspace{0.1cm}

\setcounter{tocdepth}{1}
\tableofcontents


\section{Introduction}
\label{sec:introduction}
\setcounter{equation}{0}

\subsection{Strongly magnetized plasmas}

Magnetized plasmas are encountered in a wide variety of astrophysical situations, but also in magnetic fusion devices such as tokamaks, where a large external magnetic field needs to be applied in order to keep the plasma particles on desired tracks. In numerical simulations of such devices, this large external magnetic field should be taken into account for pushing the particles, in particle methods \cite{birdsall}. However, due to the magnitude of the concerned field, this often adds a new time scale to the simulation, thus possibly a stringent restriction on the time step. In order to handle this additional timescale, one may wish to use numerical schemes whose stability is independent of the restrictive time step, and that compute approximate solutions retaining the large-scale behavior implied by the external field, even when time steps are too coarse to capture fast oscillations.

To get some first intuition, one can consider the simplest possible situation and follow the trajectory of a single particle in a constant magnetic field $\bB$ subject to no electric field. This trajectory turns out to be a helicoid along the magnetic field lines with a radius proportional to the inverse of the magnitude of $\bB$. Hence, when this field becomes very large, the particle gets trapped along the magnetic field lines. A slightly more precise description of the dynamics is that particles spin around a point on the magnetic field line, the \textit{``guiding center''}, whose velocity is smaller than the particle velocities. When electric field effects are taken into account and the magnetic field is not constant, the situation is more complicated, but still, the apparent particle velocity is smaller than the actual one and in many situations the link between the real and the apparent
velocity is well-known in terms of electromagnetic fields $\bB$ and $\bE$; see \cite{bri_hahm_07,PDFF,FR18,Krommes,Matteo-PhD,Scott_gyrokinetic} for instance.

The behavior of a plasma, constituted of a large number of charged particles, is even more complicated and may be described by the Vlasov equation coupled with the Maxwell or Poisson equations to compute the self-consistent fields. The Vlasov equation models, in essence, the evolution of a system of charged particles under the effects of external and
self-consistent fields by describing the time-evolution of the unknown $f(t,\bx,\bv)$, depending on the time $t$, the position $\bx$, and the velocity $\bv$, which represents the
distribution of particles in the phase space for each species with $(\bx,\bv) \in \RR^{d_\bx}\times \RR^{d_\bv}$ ($d_\bx,d_\bv=1,2,3$), as 
\mm{\label{eq:vlasov} 
\begin{cases}
\ds\frac{\partial f}{\partial t}\,+\,\Div_\bx(\bv f)
\,+\,\Div_\bv(\bF f)\,=\, 0\,,\\ \, \\
f(0,\cdot,\cdot)=f_0\,,
\end{cases}
}
where the force field $\bF(t,\bx,\bv)$, which  is coupled with the distribution function $f$,  makes the equation nonlinear. For instance, for the single-species Vlasov--Poisson model, the force field stems from the electric
field $\bE(t,\bx)$, \ie, it reads 
\mm{\label{poisson}
\bF(t,\bx,\bv) = \frac{q}{m}\,\bE(t,\bx)\,,\qquad \bE(t,\bx) = -\nabla_{\bx} \phi(t,\bx)\,, \qquad - \Delta_\bx\phi = \frac{\rho}{\epsilon_0}\,,
}
where $(m,q)$ are the elementary mass and charge (of one particle), $\phi(t,\bx)$ is the electric potential, $\epsilon_0$ is the electric constant, and the charge density $\rho(t,\bx)$ is given in terms of the distribution function as
\mmn{
\rho(t,\bx) := q\int_{\RR^{d_\bv}} f(t,\bx,\bv)\,\dD \bv\,.
}
When, in addition, we take into account a magnetic field $\bB(t,\bx)$, the Lorentz force applies, \ie,
$$
\bF(t,\bx,\bv) = \frac{q}{m}\,\big(\bE(t,\bx) + \bv \wedge \bB(t,\bx)\big)\,.
$$
Although the framework of our investigation can be adapted easily to the multi-species case, we shall only consider the single-species case, for the sake of simplicity, and scale parameters to $(m,q)\equiv (1,1)$.

As we are principally interested in the analysis of numerical schemes, we shall limit our focus to the quite simple case, with an external uniform magnetic field (in direction and magnitude), so to illustrate the bottom-line of the analysis with more ease. Note that, making such an assumption, we deprive ourselves of investigating phenomena such as curvature effects while we will be able to decouple dynamics in parallel and normal directions (with respect to the magnetic field), hence, we may concentrate on the particle motion in the perpendicular plane. More explicitly, in the present paper, we set 
\mmn{
\bB(t,\bx) \,\,=\,\, \frac{1}{\eps} \, \left( \begin{array}{l}0\\0\\1\end{array}\right)\,, 
}
and follow only the evolution of the first two components of the Cartesian space, 
$\bx=(x_1,x_2)\in \RR^2$. The parameter $\eps$ is related
to the ratio between the reciprocal Larmor frequency and the advection
timescale; see \cite{PDFF,haz_mei_03} and references therein for more
details {on such a scaling.} We are particularly interested in the regime where $0<\eps\ll 1$ as it implies that the magnetic field is very strong, which is required to confine the plasma, practically speaking. 

Under these assumptions, the \textit{``long-time coherent behavior''} arising in plasmas submitted to a strong external and uniform magnetic field will be obtained by the following Vlasov equation:
\mm{
\label{eq:vlasov_maxwell}
\begin{cases}
\ds\eps\frac{\d f^\eps}{\d t}\,+\,\Div_\bx(\bv f^\eps)\,+\,\Div_\bv\left((\bE\,-\, \frac{1}{\eps}\,\bv^\perp) f^\eps\right)
\,=\, 0\,,\\ \, \\
f^\eps(0,\cdot,\cdot)=f_0\,,
\end{cases}
}
where the \textit{orthogonal velocity} $\bv^\perp:=(-v_2,v_1)$ can be seen as the rotation of the original velocity with the rotation matrix $\J$:
\mm{
\label{rotation:mat}
\bv^\perp\,=\,\J\,\bv\,,\qquad\qquad \J:=\begin{bmatrix}0&-1\\1&0\end{bmatrix}\,.
}

At the continuous level, considerable efforts have been made on the rigorous derivation of reduced models from kinetic transport equations like \eqref{eq:vlasov_maxwell}, \ie, {in the so-called \textit{oscillatory limit} $\eps\to 0$  in which, there are very fast temporal oscillations in the plasma, of time scale $\cO(\eps^2)$, in the orthogonal direction
to the magnetic field}; see \cite{HanKwan_PhD,Lutz_PhD,Herda_PhD,HerdaR,FR18} for relatively recent panoramas on the question. In fact, one can obtain this limit system either using the
 formal Hilbert or Poincar\'{e} expansion (as in \cite{FR16,FR17}) or  by a rigorous approach (only when the magnetic field is homogeneous in space), \cf \cite{fre_son_98,GSR:99,SR:02} or more recently using the characteristic curves \cite{miot2016gyrokinetic}.  Nonetheless, the most remarkable mathematical
result is restricted to the two-dimensional setting with a constant magnetic field and
with interactions described through the Poisson equation, and yet validates
only half of the slow dynamics; see \cite{SR:02}, which is built on
\cite{GSR:99} and recently revisited in
\cite{miot2016gyrokinetic}. In fact, the reduced  limit system for the weak limit $f^\eps\rightharpoonup f$ writes \cite{fre_son_98}
\mm{\label{e:fgyro}
\ds\frac{\d f}{\d t}\,-\bE^\perp\cdot\nabla_\bx f \,-\frac{1}{2}\Delta\phi\,\bv^\perp\cdot\nabla_\bv f
\,=\, 0\,, \qquad\qquad (\bx,\bv)\in\RR^4\,,
}
while the limit of charge density converges to a solution of the following \cite{SR:02}:
\be\label{e:Vgyro}
\frac{\d \rho}{\d t}\,-\,\nabla_\bx\cdot\big(\rho\, \bE^\perp \big)
\,=\, 0\,, \qquad \bx\in\RR^2\,.
\ee
For the three dimensional linear Vlasov equation with an applied and smooth electromagnetic field, we  refer to \cite{FR18} for a recent study

From the discrete point of view, we are seeking methods which are able to capture this singularly oscillatory limit, {so that the numerical method provides a consistent discretization of the limit system as $\eps\to 0$, a concept known as	 \textit{asymptotic consistency}, with the numerical parameters to be independent of the singular scaling parameter $\eps$.} This concept has been widely studied for dissipative systems, since the pioneering works of \cite{jin:99,Klar:98}, in the framework of \textit{asymptotic preserving (AP) schemes}; see also the review paper \cite{jin2010asymptotic}. 
In the design of well-adapted numerical schemes to capture the slow part of the dynamics with a rather coarse discretization (compared to the fast scales), one could mention the two-scale convergence method \cite{fre_son_97,fre_son_98}, the micro-macro decomposition in \cite{CFHM:13}, lifting with multiple time variables \cite{CLM:13,CLMZ:MMS17,CLMZ:JCP17,CCLMZ:19}, exponential integrators \cite{FHLS:15,FHS:15}, frequency filtering \cite{HLW:19}, and implicit-explicit time discretizations \cite{FR16,FR17,filbet-yang}. The reader is also referred to \cite{CCZ:18,CHZ:18} for some numerical comparisons, including comparisons with more standard methods.

In the present work, we investigate the strategy of a series of work \cite{FR16,FR17,filbet2018numerical}, which live in the context of particle methods \cite{Lee}, and  hinge upon investigating the characteristics of the system, instead of using directly the PDE. Our goal, indeed, is to provide a complete convergence analysis of the Particle-In-Cell (PIC) methods introduced in \cite{FR16}, solving for the following system of characteristics:
\mm{
\label{e:Newton}
\begin{cases}
\eps\,\bx_\eps'(t)=\ds{\bv_\eps(t)}\,,
\\[1em]
\ds
\eps\,\bv_\eps'(t)={\bE(t,\bx_\eps(t))}\,-\,\frac{\bv_\eps^\perp(t)}{\eps}\,, 
\\[1em]
\bx_\eps(s)=\bx_\eps^s, \qquad \bv_\eps(s)=\bv_\eps^s\,,
\end{cases}
}
for $t\geq 0$ and for any regime of the scaling parameter $\eps$. So far, and in \cite{FR16}, some well-adapted schemes have been designed and analyzed in the regime where $\eps^2 $ is much smaller than the time step of the numerical scheme. Here, we shall perform a complete convergence and asymptotic error analysis for any values of the asymptotic parameter $\eps$ and of the time step, denoted $\Dt$ in the sequel.

In order to performing such an analysis, we start with estimates for the continuous model in \S\ref{sec:2}, followed by the discrete estimates for two versions of first-order numerical schemes in \S\ref{sec:3} and \S\ref{sec:4}. Then, we establish the convergence analysis for a second-order L-stable method in \S\ref{sec:5}, and we provide some numerical illustrations in \S\ref{sec:num_exp}.

To carry out a complete and rigorous analysis, details of the system and the schemes obviously come into play but some enlightening insights on the final outcomes may be obtained by considering an abstract system. We conclude by providing the reader with such abstract considerations in Appendix~\ref{s:app}.

\begin{notation}
{Hereinafter, and for the sake of brevity, we use}
\begin{itemize}
\item $\lesssim A$ to denote $\leq c_0 A$ for some universal constant $c_0$, and
\item $\underset{{\scriptscriptstyle {\alpha}}}{\lesssim} A$ to denote $\leq c_0(\alpha) A$ for some constant $c_0$ depending on $\alpha$.
introduced in \cite{FR16}, for the two-dimensional system with a homogeneous external magnetic field.
\end{itemize}\
\end{notation}

\begin{notation}
Our estimates shall be expressed in terms of
$$
K_0\ :=\ \|\bE\|_{L^\infty}\,,\quad 
K_t\ :=\ \left\|\partial_t\bE\right\|_{L^\infty}\,,\quad 
K_{x}\ :=\ \|\dD_\bx\bE\|_{L^\infty}\,,\quad 
K_{xx}\ :=\ \|\dD_\bx^2\bE\|_{L^\infty}\,,\quad\cdots 
$$
{In particular,} we assume global bounds on the electric field and its derivatives. We expect that a counterpart could be obtained  when fields are only locally bounded but the initial density is compactly supported. We omit, however, to state and prove this variant of our results as required adaptations are expected to be quite technical but rather classical. 
\end{notation}

\begin{remark}
  Part of our motivation to make explicit the dependence on $\bE$ in our estimates comes from the will to reduce the gap in extrapolating our results to a more nonlinear context where field equations couple fields to densities. In this respect, it is crucial to note that each time derivative of $\bE$ leads to an extra $\eps$-factor, since in a nonlinear context one expects to prove uniform $L^\infty$-bounds only on $(\eps\d_t)^\alpha\dd_\bx^{\,\beta} \bE$ and not on $\d_t^\alpha\dd_\bx^{\,\beta} \bE$ itself.
\end{remark}

\section{Oscillatory limit of the continuous model}
\label{sec:2}

In this section, we consider the characteristic system \eqref{e:Newton} of the Vlasov equation \eqref{eq:vlasov_maxwell} in the oscillatory limit, which corresponds to the limit system \eqref{e:Vgyro}. {It is worth mentioning that our presentation of the asymptotic analysis at the continuous level is close to \cite[\S 3]{FR18}, though with a different scaling for the time. Despite this similarity,  we prefer to expound this continuous analysis mainly because, later on, and at the discrete level,  we have to go along similar lines  for the numerical method.} 

We aim to compare the characteristic system
\eqref{e:Newton} with the limit system as $\eps\to 0$, that is, the \textit{guiding-center approximation}, obtained as the solution of the following equation:
\be
\label{e:Ngyro}
\begin{cases}
\bx'(t)\,=\,-\bE^\perp(t,\bx(t)), \qquad \forall t\geq 0,\\[1em]
\bx(s)=\bx^s\,.
\end{cases}
\ee
In particular, in this section, we prove $\bx=\displaystyle\lim_{\eps\to 0}\bx_\eps$ when $(\bx_\eps,\bv_\eps)$ solves \eqref{e:Newton} and $\bx$ solves \eqref{e:Ngyro}, and we quantify the corresponding error estimate. 

{To work with} a quantity that is slower than $\bx_\eps$, we introduce
$\by_\eps$ defined as 
\begin{equation}
\label{eq:yeps}
\by_\eps(t):=\bx_\eps(t)-\eps\,\bv_\eps^\perp(t)\,,
\end{equation}
where the couple $(\bx_\eps,\bv_\eps)$ is the solution to the
characteristic curves \eqref{e:Newton} of the kinetic model \eqref{eq:vlasov_maxwell}. Note that, hereinafter and for later use, we denote by 
\[(\bX_\eps(t,s,\bx_\eps^s,\bv_\eps^s),\bV_\eps(t,s,\bx_\eps^s,\bv_\eps^s)):=(\bx_\eps(t),\bv_\eps(t))\]
the value of the solution to \eqref{e:Newton}, at any time $t$, and accordingly $\bY_\eps(t,s,\bx_\eps^s,\bv_\eps^s):=\by_\eps(t)$. Likewise, when $\bx$ solves \eqref{e:Ngyro}, we denote $\bX(t,s,\bx^s):=\bx(t)$. 

As a preliminary remark, we note that solutions to \eqref{e:Newton}
are global in time as soon as $\bE\in L^\infty$. Also, we fix classical
notation $W^{k,p}$ for the Sobolev space with derivatives up to order $k$,
measured in the $L^p$ norm. We shall use, hereinafter, canonical Euclidean $\ell_2$-norms on vectors and corresponding operator norms on linear operators and matrices. Then, we have the following result:

\begin{theorem}
\label{thm:second_estim}
\begin{enumerate}[(i)] 
\item Assume that $\bE\in W^{1,\infty}$. Then, for any $\eps>0$, for the difference between flows of \eqref{e:Newton} and \eqref{e:Ngyro}, we have 
\begin{align*}
\|\bX_\eps(t,0,\bx_\eps^0,\bv_\eps^0)-\bX(t,0,\bx_\eps^0)\|
&\,\underset{{\scriptscriptstyle {\bE}}}{\lesssim}\,
\eps\,e^{K_{x}\,t}(1+t^2)\,(\|\bv^0_\eps\|+\eps)\,.
\end{align*}
\item Assume that $\bE\in W^{2,\infty}$. Then, for any $\eps>0$, we have
\begin{align*}
\|\bY_\eps(t,0,\bx_\eps^0,\bv_\eps^0)
-\bX(t,0,\bx^0_\eps-\eps(\bv^0_\eps)^\perp)\| 
&\,\underset{{\scriptscriptstyle {\bE}}}{\lesssim}\, \eps^2\,(1+t^4)\,e^{2\,K_{x}\,t}\left(1+\eps^2+\|\bv^0_\eps\|^2\right)\,.
\end{align*}
\end{enumerate}
\end{theorem}

\begin{remark}
The reader may rightfully remark that the foregoing estimates do not
scale sharply when $t\to0$ or $t\to\infty$. For instance, as the left-hand
side vanishes at time $t=0$, one could expect that the right-hand vanishes as well. Indeed, one may simply resolve such an issue by changing the $\cO(\eps)$ bound (for the first case) into $\cO(\min(\eps,t/\eps))$, by taking into account direct bounds on time derivatives. However, we have chosen to disregard this refinement as this provides an improvement only when $t=\cO(\eps^2)$. Also, in the reverse direction $t\to\infty$, we have chosen not to optimize constants or power of times. Typically, for the sake of simplicity, we have chosen to use bounds such as  
\[
e^{K_xt}\,\int_0^t e^{-K_xs}\dd s\,=\,\frac{e^{K_xt}-1}{K_x}
\,\leq\,t\,e^{K_x\,t}\,.
\]
\end{remark}

Before detailing the proof of Theorem~\ref{thm:second_estim}, we would like to highlight that, in essence, there are two steps in the proof :
\begin{enumerate}[(i)]
\item The first step is to prove the boundedness of the solutions of the characteristic system \eqref{e:Newton} with respect to $\eps$, sometimes referred to as $\eps$-boundedness below. We will discuss this kind of $\eps$-uniform estimates in \S\ref{sec:2.1}. Note that rough direct bounds would predict blow up in terms of $\eps$, not because of the skew-symmetric $\eps^{-2}$ term but due to existing $\eps^{-1}$ terms.
\item In the second step, one derives, from system~\eqref{e:Newton}, that the function to be compared with, either $\bx_\eps$ or $\by_\eps$, satisfies an equation, which is asymptotically close to the expected limiting equation \eqref{e:Ngyro}, the guiding center equation. This step is algebraic in nature and is carried out in \S\ref{sec:2.2}. It builds upon the fact that the velocity equation in the characteristic system~\eqref{e:Newton} yields that, formally speaking, $\bv_\eps$ is the sum of a quantity of size $\eps$ and the time derivative of size $\eps^2$. Note that the first step precisely ensures that this formal reasoning is valid.
\end{enumerate}
The conclusion is then obtained, also in \S\ref{sec:2.2}, by combining the two steps with a stability estimate on the expected limiting equation. 

\subsection{Uniform estimates on characteristics}
\label{sec:2.1}
In order to establish uniform estimates with respect to $\eps$, we, firstly, define  a new variable called $\bz_\eps$, as \footnote{Note that the definition of
  $\bz_\eps$ differs from the one in \cite{FR16} only by a scaling of $\eps$.}
\begin{equation}
\label{def:z}
\bz_\eps(t) \,:=\,\bv_\eps(t)\,+\,\eps \,\bE^\perp(t,\bx_\eps(t)), \quad t\geq 0\,,
\end{equation}
whose temporal dynamics is more purely oscillatory than the one of $\bv_\eps$, in the sense that the influence of non-oscillatory terms is less significant, namely here $\bz_\eps'+\eps^{-2}\bz_\eps^\perp=\cO(1)$ whereas $\bv_\eps'+\eps^{-2}\bv_\eps^\perp=\cO(\eps^{-1})$. This can be seen explicitly, since the time evolution of $\bz_\eps$ (\cf  \cite[eq. (5)]{FR16}) obeys
\begin{align}
\label{eq:z_time}
\begin{split}
\bz_\eps'(t)\,&=\,-\frac{1}{\eps^2}\bz_\eps^\perp(t)
\,+\,\eps\,\frac{\dD}{\dD t} \bE^\perp(t,\bx_\eps(t))\,,\\
&=\,-\frac{1}{\eps^2}\,\bz_\eps^\perp(t)
\,+\,\eps\,\d_t\bE^\perp(t,\bx_\eps(t))\,+\,\dD_\bx\bE^\perp(t,\bx_\eps(t))\left(\bz_\eps(t)\,-\,\eps\,\bE^\perp(t,\bx_\eps(t))\right)\,,
\end{split}
\end{align}
by using $\eps\,\bx'_\eps(t)={\bv_\eps(t)}$ from \eqref{e:Newton}. We should emphasize that the motivation for introducing $\bz_\eps$ and looking for a dynamics as purely oscillatory as possible is that the oscillatory part of the evolution preserves the Euclidean norm, hence one obtains readily a good estimate by using $\bz_\eps$, as in the following lemma. 

\begin{lemma}
\label{lem:apriori}
Let us assume that $\bE\in W^{1,\infty}$ and consider
the system corresponding to the characteristic curves of the Vlasov
equation \eqref{e:Newton}. Then, the auxiliary variable $\bz_\eps$ introduced in \eqref{eq:z_time} and the velocity $\bv_\eps$ are bounded {as}
$$
\begin{cases}
 \ds \|\bz_\eps(t)\|
\,\leq\,
e^{K_{x}\,t}
\big(\|\bv^0_\eps\|+\eps\,K_0\big)
+\eps\,t\,e^{K_{x}\,t}\left( K_t \,+\,K_{x}\,K_0 \right)\,, 
\\[0.9em]
\ds\| \bv_\eps(t)\| \,\leq\, \|\bz_\eps(t)\|\,+\,K_0\,\eps\,.
\end{cases}
$$
\end{lemma}

\begin{proof}
By taking the scalar product of \eqref{eq:z_time} with $\bz_\eps(t)$, which cancels out the singular term $-\bz_\eps^\perp(t)/\eps^2$, we obtain
\begin{align*}
 \frac{1}{2}\frac{\dD}{\dD
  t}\,\|\bz_\eps(t)\|^2&=\,\eps\,\bz_\eps(t)\cdot\d_t
                       \bE^\perp(t,\bx_\eps(t))\\ 
&\quad+\,\bz_\eps(t)\cdot\dD_\bx
  \bE^\perp(t,\bx_\eps(t))\,\bz_\eps(t)\,-\,\eps\,\bz_\eps(t)\cdot(\dD_\bx
  \bE^\perp \bE^\perp)(t,\bx_\eps(t))
\\
&\leq \eps \,K_t\,\|\bz_\eps(t)\| \,+\,K_{x} \,\|\bz_\eps(t)\|^2\,+\,\eps\,K_{x}\,K_0\,\|\bz_\eps(t)\|\,.
 \end{align*}
Denoting by $t_0$ the supremum of times in $[0,t]$ where $\bz_\eps$ vanishes\footnote{By definition $t_0=0$ if $\bz_\eps$ does not vanish.}, one may simplify by $\|\bz_\eps\|$ on $(t_0,t)$ to derive for any $s\in(t_0,t)$
\begin{align*}
 \frac{\dD}{\dD t}\|\bz_\eps\|(s)\leq \,K_{x}\,\|\bz_\eps(s)\|+\eps\left( K_t\,+\,K_{x}\,K_0\right)\,,
 \end{align*}
which, by integration, it yields
\begin{align*}
 \|\bz_\eps(t)\|&\leq\, e^{K_{x} (t-t_0)}
\|\bz_\eps(t_0)\|+\eps\,(t-t_0)\,e^{K_{x}\,(t-t_0)}
\,\left(K_t\,+\, K_{x}  \,K_0\right)\,,\\
&\leq\, e^{K_{x} t}
\|\bz_\eps(0)\|+\eps\,t\,e^{K_{x}\,t}\,\left(K_t\,+\, K_{x}  \,K_0\right)\,.
\end{align*} 
The second estimate, on $\bv_\eps(t)$, follows readily from $\bv_\eps(t)=\bz_\eps(t)\,-\,\eps \bE^\perp(t,\bx_\eps(t))$.
\end{proof}

So, one concludes that provided that the electric field is regular enough, \eg, $\bE\in W^{1,\infty}$, the norms $\|\bz_\eps\|$ and $\|\bv_\eps\|$ are bounded locally in time, uniformly with respect to $\eps$. 

\subsection{Proof of Theorem~\ref{thm:second_estim}}
\label{sec:2.2}

Now, we derive from the second equation of \eqref{e:Newton}
\begin{align*}
\eps\,(\bv^\perp_\eps)'(t)\,-\,\bE^\perp(t,\bx_\eps(t)) \,=\,\frac{\bv_\eps(t)}{\eps}\,,
\end{align*}
which, combined with the first equation of \eqref{e:Newton}, yields
\be\label{e:y-aux}
(\bx_\eps-\eps\bv_\eps^\perp)'(t)\,=\,-\bE^\perp(t,\bx_\eps(t))\,.
\ee
This shows that $\bx_\eps$ satisfies an equation seemingly close to the guiding center equation \eqref{e:Ngyro}.

Then, in order to prove the first estimate of Theorem~\ref{thm:second_estim}, we subtract \eqref{e:y-aux} from the limit system in \eqref{e:Ngyro}, integrate over $[0,t]$, and use the Lipschitz bound on $\bE$ as well as  Lemma~\ref{lem:apriori}, to obtain
\begin{align}\label{eq:asymp_dif}
\begin{split}
\|\bx_\eps(t)-\bX(t,0,\bx_\eps^0)\|
&\leq \eps(\|\bv_\eps^0\|+\|\bv_\eps(t)\|)
+K_x\int_0^t\|\bx_\eps(s)-\bX(s,0,\bx_\eps^0)\|\,\dD s\\
&\leq K_x\int_0^t\|\bx_\eps(s)-\bX(s,0,\bx_\eps^0)\|\,\dD s\\
&\quad+\eps\|\bv_\eps^0\|
+\eps\,\left(
K_0\eps
+e^{K_{x}\,t}(\|\bv^0_\eps\|+\eps\,K_0)
+\eps\,t\,e^{K_xt}\,\left( K_t \,+\,K_{x}\,K_0 \right)\right)\,.
\end{split}
\end{align}

At this stage, we are ready to apply \textit{Gr\"onwall's lemma} (in the integral form): Let $A$, $a$, and $K$ be non-negative constants such that
\begin{subequations}
\mm{\label{eq:gronwall1}
\qquad A(t)\leq a(t)+K\int_0^t A(s)\,\dD s, \qquad \forall t\geq 0.
}
Then, it holds
\mm{\label{eq:gronwall2}
\qquad A(t)\leq a(t)+K\int_0^t e^{K(t-s)}a(s)\,\dD s\,, \qquad \forall t\geq 0.
}
\end{subequations}
For applying this lemma to the bound of $\|\bx_\eps(t)-\bX(t,0,\bx_\eps^0)\|$ in \eqref{eq:asymp_dif}, we make use of crude estimates	 for simplicity:
\begin{align*}
\int_0^t e^{K(t-s)}s^a\,e^{Ks}\,\dD s&\leq e^{Kt}\,t^{a+1}\,,
&\int_0^t e^{K(t-s)}s^a\,\dD s&\leq e^{Kt}\,t^{a+1}\,,
\end{align*}
Thus, one gets
\begin{align*}
\|\bx_\eps(t)-\bX(t,0,\bx_\eps^0)\|
&\lesssim
\eps\,e^{K_{x}\,t}(1+t)\Big(\|\bv^0_\eps\|+\eps\,K_0
+\eps\,t\,\left( K_t \,+\,K_{x}\,K_0 \right)\Big)\,.
\end{align*}
This concludes the proof of the first part of Theorem~\ref{thm:second_estim} (using that  $t \lesssim(1+t^2)$). 

Now, in terms of the new variable $\by_\eps:=\bx_\eps-\eps\,\bv_\eps^\perp$, defined in \eqref{eq:yeps}, equation \eqref{e:y-aux} writes
\begin{equation}
\label{eq:y}
\begin{cases}
\by_\eps'(t)\,=\,-\bE^\perp(t,\by_\eps(t)+\eps \bv^\perp_\eps(t))\,, \qquad \forall t\geq 0,
\\[0.9em]
\ds\by_\eps(0) = \bx^0_\eps - \eps (\bv^0_\eps)^\perp\,.
\end{cases}
\end{equation}
By Taylor expansion (with integral remainder) we obtain
\be\label{e:y-expand}
\by_\eps' \,=\,
-\bE^\perp(t,\by_\eps)\,-\,\eps\,\dD_\bx\bE^\perp(t,\by_\eps)\,\bv_\eps^\perp\,+\,\eps^2\,\Theta_\eps(t,\by_\eps,\bv_\eps)\,,
\ee
where the remainder function $\Theta_\eps$ is bounded as
$$
\|\Theta_\eps(t,\by_\eps,\bv_\eps) \| \,\leq\, \frac12\,K_{xx}
\,\|\bv_\eps\|^2\,.
$$
Recalling that we would like to obtain an evolution equation for $\by_\eps$ that would be an $\cO(\eps^2)$-perturbation of the guiding center equation in \eqref{e:Ngyro}, the  issue, now, is to replace the variable $\bv_\eps$ on the right hand side of \eqref{e:y-expand}. For this purpose, we use, once again, the second equation of \eqref{e:Newton} written
as $\bv_\eps^\perp\,=\,-\eps^2\,\bv'_\eps\,+\,\eps\bE(t,\bx_\eps)$, and conclude that
$$
\by_\eps'\,=\,-\bE^\perp(t,\by_\eps)\,-\,\eps^2\,\left[\dD_\bx\bE^\perp(t,\by_\eps)\,
\bE(t,\bx_\eps) \,+\, \Theta_\eps(t,\by_\eps,\bv_\eps)\right]  \,+\,\eps^3 \,\dD_\bx\bE^\perp(t,\by_\eps)\,\bv_\eps' \,.
$$
The last term of the foregoing equation will be written with the help of a complete time derivative, \ie,
$$
\dD_\bx\bE^\perp(t,\by_\eps)\,\bv_\eps' \,=\,
\left(\dD_\bx\bE^\perp(t,\by_\eps)\,\bv_\eps\right)'\,-\,\left( \partial_t\dD_\bx\bE^\perp
  (t,\by_\eps) \,+\, \dD_{\bx}^2\bE^\perp (t,\by_\eps) \,\by_\eps'\right)\,\bv_\eps\,.
$$
Then, using \eqref{e:y-aux}, we obtain
\begin{eqnarray*}
\left(\by_\eps-\eps^3\dD_\bx\bE^\perp(t,\by_\eps)\bv_\eps\right)' &=&-\bE^\perp(t,\by_\eps) -\eps^2\left[\dD_\bx\bE^\perp(t,\by_\eps)
\bE(t,\bx_\eps) + \Theta_\eps(t,\by_\eps,\bv_\eps)\right]  
\\
&&-\eps^3\left(\d_t\dD_\bx\bE^\perp(t,\by_\eps)-\dD^2_\bx\bE^\perp(t,\by_\eps)\bE^\perp(t,\bx_\eps)\right)\bv_\eps\,.
\end{eqnarray*}
which is an equation $\cO(\eps^2)$-close to the equation in \eqref{e:Ngyro}.

Finally, and similarly as for the first estimate, we subtract the foregoing equation from the equation in \eqref{e:Ngyro} for a solution emanated from $\bx^0_\eps-\eps(\bv^0_\eps)^\perp$ at $s=0$, integrate over $[0,t]$, and use the Lipschitz bound on $\bE$ as well as  Lemma~\ref{lem:apriori}, to get
\begin{align*}
\|\by_\eps(t)-\bX(t,0,\bx^0_\eps-\eps(\bv^0_\eps)^\perp)\|
\leq\,&\, K_x\int_0^t\|\by_\eps(s)-\bX(s,0,\bx_\eps^0-\eps(\bv^0_\eps)^\perp)\|\,\dD s\,\\
&+\eps^2\frac{K_{xx}}{2}\int_0^t\|\bv_\eps(s)\|^2\,\dD s
+\eps^3(K_{tx}+K_{xx}K_0)\int_0^t\|\bv_\eps(s)\|\,\dD s\\
&+\eps^3\,K_x\,(\|\bv_\eps^0\|+\|\bv_\eps(t)\|)+\eps^2K_xK_0\,t.
\end{align*}
In the same line of argument as for the first estimate, that is, using Lemma~\ref{lem:apriori} and the Gr\"onwall lemma, one obtains after some manipulations
\begin{align*}
 \| \by_\eps(t)&-\bX(t,0,\bx^0_\eps-\eps(\bv^0_\eps)^\perp)\| \\
  \lesssim\, & \,\eps^2\,t\,e^{K_{x}\,t}\,(1+K_xt)\,\left[e^{K_{x}\,t}\,
     \,K_{xx}\|\bv^0_\eps\|^2 + K_x\,K_0\right]\\
&+\eps^3\,e^{K_{x}\,t}\,(1+K_xt)\,(\|\bv^0_\eps\|+\eps\,K_0)\,
\left[ \,t\,e^{K_{x}\,t}\,K_0\,K_{xx}+t\,(K_{tx}+K_{xx}K_0)+K_x\right]\,\\
&+\eps^4\,t\,e^{K_{x}\,t}\,(1+K_xt)\,(K_t\,+\,K_{x}\,K_0)
\,\left[ K_{xx}\,t^2\,e^{K_{x}\,t}\,(K_t\,+\,K_{x}\,K_0)
+t\,(K_{tx}+K_{xx}\,K_0)
+K_x\right]\,,
\end{align*}
which concludes the desired estimate by employing Young inequalities, typically, in the form
\[
\eps\,t\,\|\bv^0_\eps\|
\lesssim
\eps^2\,t^2+\|\bv^0_\eps\|^2\,.
\]

\subsection{From characteristics to PDE's: estimates on the {density}}
\label{sec:2.3}

As we will explain in this section, thanks to $L^\infty$-bounds on the characteristics system (see Lemma \ref{lem:apriori}) and its asymptotic evolution (see Theorem \ref{thm:second_estim}), it is straightforward to derive bounds on particle distributions, that is, on the densities, in the $W^{-1,1}$ topology. We recall that in the dual space $W^{-1,1}:=(W^{1,\infty})^*$, the canonical seminorm is defined by
\mmn{
\|\mu\|_{\dot{W}^{-1,1}}\,:=\,\sup_{\|\nabla\varphi\|_{L^{\infty}}\leq1}\left|\int \varphi\,\dD\mu\right|\,,
}
for $\mu\in W^{-1,1}$. Incidentally, note that the seminorm on $W^{-1,1}$ defines a distance, equivalent to the 1-Wasserstein distance, on probability measures with a finite first moment. 

{Let us also recall} the classical link between characteristics and solutions to continuity equations. In fact, the solution of an abstract continuity equation
\[
\d_t G \,+\,\Div_\bfa (\cX\,G)\,=\,0\,,
\]
with initial datum $G_0$, {writes} $G(t,\cdot)=\bA(t,0,\cdot)_*(G_0)$, where $\bA$ is the \textit{flow} associated with the differential equation $\bfa'=\cX(t,\bfa)$, and $\cA_*(\mu)$ denotes the push-forward of $\mu$ by $\cA$, which is defined by 
$$
\langle \cA_*(\mu),\varphi\rangle
:=\langle \mu,\varphi\circ\cA\rangle\,,
$$
for {all} test-functions $\varphi$. Note that the backbone of particle methods is the fact that the push-forward of a Dirac mass is given by $\cA_*(\delta_{\bfa_0})=\delta_{\cA(\bfa_0)}$. Note also that when $\cX$ is divergence-free, the formula {matches the one solving} the associated transport equation, namely $G(t,\cdot)=G_0\circ \bA(0,t,\cdot)$, {where $\circ$ stands for function composition,} but differs otherwise.

In particular, solutions to \eqref{eq:vlasov_maxwell} are obtained as 
\[
f^\eps(t,\cdot)=(\bX_\eps,\bV_\eps)(t,0,\cdot)_*(f_0),
\]
and, consequently, the corresponding charge density reads
\[
\rho^\eps(t,\cdot)=\bX_\eps(t,0,\cdot)_*(f_0)\,.
\]
{Having these, one can} readily derive the following estimate on {the density}.

\bc
\label{cor:density_estim}
Assume that the electric field is Lipschitz $\bE\in W^{1,\infty}$ and that $f_0$ is a probability with finite first moment. Then, for any $\eps>0$, the following estimate holds
 \begin{equation*}
\|\rho^\eps(t,\cdot)-\rho(t,\cdot)\|_{\dot{W}^{-1,1}} 
\,\underset{{\scriptscriptstyle {\bE}}}{\lesssim} 
\eps\,e^{K_{x}\,t}(1+t^2)\,
\int_{\RR^2\times\RR^2}
(\eps+\|\bv\|)\ \dD f_0(\bx,\bv)\,,
\end{equation*}
where $\rho^\eps$ is the charge density computed from $f^\eps$ solution to \eqref{eq:vlasov_maxwell} whereas $\rho$ solves to \eqref{e:Vgyro} with the same initial datum as $\rho^\eps$. 
\ec

\begin{proof}
For any test function $\varphi\in W^{1,\infty}(\RR^2)$, one gets from the formula recalled above
\begin{eqnarray*}
\left\langle \rho^\eps(t,\cdot)-\rho(t,\cdot),\varphi \right\rangle
&=& \int_{\RR^2\times\RR^2}
    \left(\varphi\left(\bX_\eps(t,0,\bx,\bv)\right) \,-\,
    \varphi\left(\bX(t,0,\bx^0)\right)\right)\,\dD f_0(\bx,\bv)\,.
\end{eqnarray*}
so that
\mmn{
\left|\left\langle \rho^\eps(t,\cdot)-\rho(t,\cdot),\varphi
  \right\rangle\right|
\leq\,\|\nabla\varphi\|_{L^\infty(\RR^2)}
\int_{\RR^2\times\RR^2}\|\bX_\eps(t,0,\bx,\bv)-\bX(t,0,\bx)\|\,\dD f_0(\bx,\bv),
}
and the proof is concluded by applying the first part of Theorem~\ref{thm:second_estim}.
\end{proof}

\section{First-order scheme on the original spatial variable}
\label{sec:3}

In the present section, we first complete the analysis of the first-order IMEX method introduced in \cite{FR16}. For a chosen constant time discretization step $\Dt$, we define a discrete time evolution, for $n\geq 0$, by
\be
\label{e:dNewton-1st}
\begin{cases}
\ds\frac{\bx_\eps^{n+1}-\bx_\eps^{n}}{\Dt }&=\,\ds\frac{\bv_\eps^{n+1}}{\eps}\,,
\\[1em]\ds 
\eps\,\frac{\bv_\eps^{n+1}-\bv_\eps^{n}}{\Dt }&=\,\ds
\bE(t_n,\bx_\eps^n)\,-\,\frac{(\bv_\eps^{n+1})^\perp}{\eps}\,,
\end{cases}
\ee
where $t_n=n\,\Dt$. Note that the scheme \eqref{e:dNewton-1st} is semi-implicit but the implicit part (the velocity update) is linear. It only requires solving the $2\times 2$ linear system
\[\bv_\eps^{n+1}+\frac{\Dt}{\eps^2}(\bv_\eps^{n+1})^\perp=\bv_\eps^n+\frac{\Dt}{\eps}
\bE(t_n,\bx_\eps^n)\,.\]
This highlights the role played by the matrix $\Id+\lambda\,\J$, with $\lambda:=\Dt/\eps^2>0$ and $\J$ as in \eqref{rotation:mat}; see Lemma \ref{l:1st-J} for further discussion on this matrix.

{As we will see later on, in Theorem \ref{thm:error_estim}, the scheme \eqref{e:dNewton-1st} has a unique solution, which means that it allows defining the sequence $(\bx_\eps^n,\bv_\eps^n)_{n\geq 0}$, which is expected to approximate the solution $(\bx_\eps,\bv_\eps)$ of \eqref{e:Newton} at times $(t_n)_{n\geq0}$.} The foregoing scheme is designed to capture the  evolution of the space variable $\bx_\eps$ even when $\Dt\gg \eps^2$, \ie, {the asymptotic regime in which the traditional schemes are doomed to fail due to instability}. We aim to show that this {asymptotic convergence is sufficient for obtaining an $\eps$-uniform error estimate, \ie, to ensure that the convergence of the scheme for the numerical error $\|\bx_\eps(t_n)-\bx_\eps^n\|$ will be locally uniform} with respect to $t_n\in\RR^+$ and $\eps\in\RR^+$. Moreover, we will {reveal that the convergence rate can be improved on the guiding center variable $\by_\eps$ compared to $\bx_\eps$ itself, when $\eps\ll 1$. Indeed, this variable, introduced in \eqref{eq:yeps}, may be thought as a slower version of $\bx_\eps$, hence it is expectedly more  advantageous to be used in the strongly oscillatory regime. To state comparisons for the discrete dynamics, we  introduce the discrete counterpart of $\by_\eps(t)$ in \eqref{eq:yeps}, that is,}
$$
\by_\eps^n=\bx_\eps^n-\eps\,(\bv_\eps^n)^\perp\,,
$$ 
for which,  using the scheme \eqref{e:dNewton-1st}, one gets the update	
\mm{
\label{eq:dy_time}
\dfrac{\by_\eps^{n+1}-\by_\eps^{n}}{\Dt }=\,-
\bE^\perp(t_n,\by_\eps^n+\eps\,(\bv_\eps^n)^\perp)\,.
}
By taking formally the limit $\eps\to 0$, we would anticipate that, at the discrete level, the reduced asymptotic model for the limit of $\bx_\eps^n$ is the explicit Euler scheme, that is
\be
\label{e:dNgyro-1st}
\dfrac{\bx^{n+1}-\bx^{n}}{\Dt }=\,-\bE^\perp(t_n,\bx^n)\,, \quad {\rm
  for }\,\, n\geq 0\,.
\ee

Now, we can state our main theorem on the scheme \eqref{e:dNewton-1st}. 

\begin{theorem}
\label{thm:error_estim}
 {The first-order scheme \eqref{e:dNewton-1st} possesses a unique solution. Moreover 
 \begin{enumerate}[(i)]
 \item when $\bE\in W^{1,\infty}$, the space variable $\bx_\eps$ satisfies for all $n\geq0$, $\Dt>0$ and $\eps>0$,
\mmn{
\|\bx_\eps^n-\bx_\eps(t_n)\|
&\underset{{\scriptscriptstyle {\bE}}}{\lesssim}\,
(1+t_{n}^2)\,e^{2K_xt_n}(1+\|\bv_\eps^0\|+\eps)
\times\min\Bigg(\frac{\Dt}{\eps^3}(1+\eps^2),
\Dt+\eps\Bigg),
}
\item when $\bE\in W^{2,\infty}$, the guiding center variable $\by_\eps$ satisfies for all $n\geq0$, $\Dt>0$ and $\eps>0$, 
\mmn{
\|\by_\eps^n-\by_\eps(t_n)\|
&\underset{{\scriptscriptstyle {\bE}}}{\lesssim}\,
(1+t_{n}^4)\,e^{2K_xt_n}(1+\|\bv_\eps^0\|^2+\eps^2)
\times\min\Bigg(\frac{\Dt}{\eps^3}\,(1+\eps)\,,\Dt+\eps^2\,\Bigg)\,.
}
\end{enumerate}}
\end{theorem}

\begin{corollary}\label{cor:error_estim}
{Theorem \ref{thm:error_estim} implies that, for different regimes of $\Dt$ and $\eps$,} 
\mm{
\|\bx_\eps^n-\bx_\eps(t_n)\| \,\underset{{\scriptscriptstyle (t_n,\bE,\|\bv_\eps^0\|)}}{\lesssim}\,
(1+\eps)
\begin{cases}
\quad\frac{\Dt }{\eps^3}(1+\eps^2),&\qquad \textrm{if }\quad\Dt\lesssim\eps^4\,,\\
\quad\eps,&\qquad\textrm{if }\quad\eps^4\lesssim \Dt\lesssim \eps\,,\\
\quad\Dt,&\qquad\textrm{if }\quad\eps\lesssim \Dt\,,
\end{cases}
}
and
\mm{
\|\by_\eps^n-\by_\eps(t_n)\|\,\underset{{\scriptscriptstyle (t_n,\bE,\|\bv_\eps^0\|)}}{\lesssim}\,
(1+\eps^2)\,
\begin{cases}
\quad\frac{\Dt }{\eps^3}(1+\eps),&\qquad \textrm{if }\quad\Dt\lesssim\eps^5\,,\\
\quad\eps^2,&\qquad\textrm{if }\quad\eps^5\lesssim \Dt\lesssim \eps^2\,,\\
\quad\Dt,&\qquad\textrm{if }\quad\eps^2\lesssim \Dt\,.
\end{cases}}
\end{corollary}

\begin{remark}
The estimates of Corollary \ref{cor:error_estim} shows that for a fixed $\Dt$, when $\eps$ is either sufficiently small or sufficiently large, both estimates {boil down to} a  uniform $\cO(\Dt)$ estimate. Nonetheless, the worst-$\eps$ scenarios yield the following uniform rates :
\[
\|\bx_\eps^n-\bx_\eps(t_n)\| \,\underset{{\scriptscriptstyle (t_n,\bE,\|\bv_\eps^0\|)}}{\lesssim}\,(\Dt)^{1/4}\,,\qquad\qquad
\|\by_\eps^n-\by_\eps(t_n)\|\,\underset{{\scriptscriptstyle (t_n,\bE,\|\bv_\eps^0\|)}}{\lesssim}\,(\Dt)^{2/5}\,,
\]
obtained, respectively, when $\Dt\sim \eps^4$ and when $\Dt\sim\eps^5$.
\end{remark}

{Based on the foregoing estimates, it will be painless to obtain an error estimate  at the particle density level. Note, however, that the full PIC error estimates would involve errors not due to the time discretizations and, hence, not taken into account here. To state the corresponding corollary, we denote the discrete flow $(\bX_\eps^{\Dt},\bV_\eps^{\Dt})(t_n,t_s,\bx^s,\bv^s)$ as the solution to \eqref{e:dNewton-1st} starting from $(\bx^s,\bv^s)$ at index $s$. Note in particular that $n\geq s\geq0$ and $t_n=n\Dt$, $t_s=s\Dt$. For later use, we also set 
\[ \bY_\eps^{\Dt}(t_n,t_s,\bx^s,\bv^s):=(\bX_\eps^{\Dt}-\eps(\bV_\eps^{\Dt})^\perp)(t_n,t_s,\bx^s,\bv^s).\]
}

\bc
\label{cor:density_error-estim}
Assume that the electric field is $\bE\in W^{1,\infty}$ and that $f_0$ is a probability with finite first moment. Then, for any $\eps>0$, the following uniform error estimate holds
 \begin{equation*}
\|\rho_{\Dt}^\eps(t_n,\cdot)-\rho^\eps(t_n,\cdot)\|_{{W}^{-1,1}} 
\,\underset{{\scriptscriptstyle {\bE}}}{\lesssim} 
\min\Bigg(\frac{\Dt}{\eps^3}(1+\eps^2),
\Dt+\eps\Bigg)\,e^{2K_{x}\,t_n}(1+t_n^2)\,
\int_{\RR^2\times\RR^2}\hspace{-1.5em}
(1+\eps+\|\bv\|)\ \dD f_0(\bx,\bv)\,,
\end{equation*}
where $\rho^\eps$ is the charge density computed from $f^\eps$ solution to \eqref{eq:vlasov_maxwell} whereas $\rho_{\Dt}^\eps$ is defined at times $(t_n)_{n\in\N}=(n\Dt)_{n\in\N}$ by
\[
\rho_{\Dt}^\eps(t_n,\cdot)\,=\,\bX_\eps^{\Dt}(t_n,0,\cdot,\cdot)_*(f_0)\,.
\]
\ec

As suggested by the presence of the $\min$ function in claimed
estimates, the proof of Theorem~\ref{thm:error_estim}, provided in
subsequent subsections, combines two kinds of estimates: {the
  $\Dt/\eps^3$ part which arises from {a detailed version} of
  classical convergence estimates (Proposition \ref{prop:E1_y} below),
  and $\eps$-uniform part which stems from combining three estimates
  through the triangle inequality, namely, asymptotic estimates at
  continuous (Theorem \ref{thm:second_estim}) and discrete
  (Proposition \ref{prop:disct_err} below) levels, and a classical
  convergence estimate for the non-stiff reduced asymptotic model
  ($\eps$-independent), which will be discussed in Proposition
  \ref{prop:E1_gc} below.} Thus, eventually,  {it leads to an error estimate which is} $\cO(\eps+\Dt)$, \ie, 
\[
\|\bx_\eps(t_n)-\bx_\eps^n\|
\,\leq\,
\|\bx_\eps(t_n)-\bx(t_n)\|
+\|\bx(t_n)-\bx^n\|
+\|\bx^n-\bx_\eps^n\|
\underset{{\scriptscriptstyle (t_n,\bE,\|\bv_\eps^0\|)}}{\lesssim} \eps+\Dt+\eps\,,
\]
with $\bx(t_n)$ solving \eqref{e:Ngyro} and $(\bx^n)_{n\geq0}$ solving \eqref{e:dNgyro-1st}.

\subsection{Direct convergence estimates}
\label{sec:3.2}
{In this section, we discuss direct convergence estimates, \ie, we estimate the difference of the numerical solution and the exact one, for the $\eps$-dependent system (in Proposition \ref{prop:E1_y}) and the asymptotic model (in Proposition \ref{prop:E1_gc}). The former estimate is called, hereinafter, the \textit{direct estimate} and gives rise to the  $\Dt/\eps^3$ part in Theorem \ref{thm:error_estim}. }
\subsubsection{$\eps$-dependent direct estimate} 
{We begin with the direct estimate for which} {it would be more convenient} to carry out the analysis in variables $(\by_\eps,\bv_\eps)$ rather than $(\bx_\eps,\bv_\eps)$. {The first step is the direct consistency analysis; we introduce consistency errors, using the $\by_\eps$-update and \eqref{e:dNgyro-1st}, as}
\begin{subequations}
\mm{
\label{eq:tauy_1st}
\bftau_{\by}^n&\,:=\,
\ds\frac{\by_\eps(t_{n+1})-\by_\eps(t_{n})}{\Dt}\,+\,\bE^\perp(t_n,\by_\eps(t_n)\,+\,\eps\,\bv_\eps^\perp(t_n))\,,
\\
\label{eq:tauv_1st}
\bftau_{\bv}^n&\,:=\,
\ds\frac{\bv_\eps(t_{n+1})-\bv_\eps(t_{n})}{\Dt}\,-\,\dfrac{\bE(t_n,\by_\eps(t_n)+\eps\, \bv_\eps^\perp(t_n))}{\eps}\,+\,\dfrac{\bv_\eps^\perp(t_{n+1})}{\eps^2}\,.
}
\end{subequations}	
These local truncation errors can be bounded as in the following lemma.

\begin{lemma}\label{l:1st_consistency}
For every $\eps>0$ and $\Dt>0$, {the truncation errors $(\bftau_{\by}^n)_{n\geq 0}$ and $(\bftau_{\bv}^n)_{n\geq 0}$ defined in \eqref{eq:tauy_1st}--\eqref{eq:tauy_1st} are bounded as}
{\begin{align*}
\|\bftau_{\by}^n\|&\underset{{\scriptscriptstyle {\bE}}}{\lesssim}
\frac{\Dt}{\eps}\ e^{K_{x}\,t_{n+1}}\ 
\Big( \|\bv^0_\eps\|+\eps\,(1+t_{n+1}) \Big)\,,\\
\|\bftau_{\bv}^n\|&\underset{{\scriptscriptstyle {\bE}}}{\lesssim}
\frac{\Dt}{\eps^4}\ e^{K_{x}\,t_{n+1}}\ (1+\eps^2)\,
\Big(\|\bv^0_\eps\|+\eps\,(1+t_{n+1})\Big)\,.
\end{align*}}
\end{lemma}
\begin{proof}
The truncation error $\bftau_{\by}^n$ can be written as
\[
\bftau_{\by}^n\,=\,
\frac{\by_\eps(t_{n+1})-\by_\eps(t_{n})}{\Dt}-\by_\eps'(t_n)\,.
\]
{Thus, for any $n\geq0$,  $\|\bftau_{\by}^n\|\leq \tfrac12\Dt\,\max_{[t_n,t_{n+1}]}\|\by_\eps''\|$, where the second derivative reads}
\[
\by_\eps''(t)\,=\,
-\d_t\bE^\perp(t,\bx_\eps(t))-\frac1\eps\dD_\bx\bE^\perp(t,\bx_\eps(t))\bv_\eps(t)\,,
\]
which, thanks to Lemma~\ref{lem:apriori}, yields the first bound:
\mmn{
\|\bftau_{\by}^n\|
&\leq \frac{\Dt}{2}
\left(
(K_t+K_xK_0)(1+K_x\,t_{n+1}\,e^{K_{x}\,t_{n+1}})+\frac{K_x}{\eps}
e^{K_{x}\,t_{n+1}}\big(\|\bv^0_\eps\|+\eps\,K_0\big)
\right)\,,
\\
&\underset{{\scriptscriptstyle {\bE}}}{\lesssim}
\frac{\Dt}{\eps}\ e^{K_{x}\,t_{n+1}}\ 
\big( \|\bv^0_\eps\|+\eps\,(1+t_{n+1}) \big)\,.
}

Likewise, for the second estimate, one has
\mmn{
\bftau_{\bv}^n=
\ds\frac{\bv_\eps(t_{n+1})-\bv_\eps(t_{n})}{\Dt}\,-\,\bv_\eps'(t_{n})
\,+\,\dfrac{\bv_\eps^\perp(t_{n+1})-\bv_\eps^\perp(t_{n})}{\eps^2}\,,
}
{which implies the estimate}
\mmn{
\|\bftau_{\bv}^n\|\leq
\frac{\Dt}{2}\,\max_{[t_n,t_{n+1}]}\|\bv_\eps''\|
\,+\,\dfrac{\Dt}{\eps^2}\,\max_{[t_n,t_{n+1}]}\|\bv_\eps'\|\,.
}
{Moreover, the second order derivative reads}
\[
\bv_\eps''(t)\,=\,
\frac1\eps\d_t\bE(t,\bx_\eps(t))+\frac{1}{\eps^2}\dD_\bx\bE(t,\bx_\eps(t))\bv_\eps(t)-\frac{1}{\eps^3}(\bE(t,\bx_\eps(t)))^\perp
-\frac{1}{\eps^4}\bv_\eps(t)\,,
\]
which concludes the proof  {combined with} Lemma~\ref{lem:apriori}, \ie, 
\mmn{
\|\bftau_{\bv}^n\|
&\,{\lesssim}\,\,\frac{\Dt}{\eps}\,e^{K_{x}\,t_{n+1}} \left(
  (K_t+K_xK_0)\,\left[1+\left(K_x+\frac{1}{\eps^2}\right)\,t_{n+1}
  \right] \,+\,  \left(K_0 + \frac{\|\bv^0_\eps\||}{\eps}\right)\,\left[ K_x + \frac{{1}}{\eps^2}\right]\right) 
\nonumber\\
&\underset{{\scriptscriptstyle {\bE}}}{\lesssim}
\,\,\frac{\Dt}{\eps^4}\ e^{K_{x}\,t_{n+1}}\ (1+\eps^2)\,
\Big( \|\bv^0_\eps\|+\eps\,(1+t_{n+1})\Big)\,.
\nonumber
}
\end{proof}

In order to obtain the $\Dt/\eps^3$ bounds, the missing piece is the stability analysis. So, we aim to investigate the stability of the implicit part, with the corresponding matrix $\Id+\lambda\,\J$ to be inverted. The following stability result is based on the fact that $\J$ is skew-symmetric and $\J^2=-\Id$.

\begin{lemma}\label{l:1st-J}
Let $\J$ be as in \eqref{rotation:mat}. Then, for any $\lambda>0$, $\Id+\lambda\J$ is invertible and
\mmn{
\left(\Id+\lambda\,\J\right)^{-1}\,=\,
\frac{1}{1+\lambda^2}\left(\Id-\lambda\,\J\right),
}
{whose norm is bounded like}
\mm{
\label{landa}
\Lambda_\lambda^{-1}:= \left\|(\Id+\lambda\,\J)^{-1}\right\|=\frac{1}{\sqrt{1+\lambda^2}}<\min\left(1,\frac1\lambda\right)\,.
}
\end{lemma}
\begin{proof}
The formula for the inverse is readily deduced from $\J^2=-\Id$. To compute its norm, note that, for any vector $\bfa$, $\bfa$ and $\J\bfa$ are orthogonal and $\|\J\bfa\|=\|\bfa\|$. From this, follows $\|\Id-\lambda\,\J\|=\sqrt{1+\lambda^2}$, hence the formula for $\Lambda_\lambda^{-1}$. The final estimate stems from $1+\lambda^2>\max(1,\lambda^2)$.
\end{proof}

Gathering consistency estimates from Lemma~\ref{l:1st_consistency} and stability information from Lemma~\ref{l:1st-J}, we now provide direct error bounds in Proposition \ref{prop:E1_y}, corresponding to the $\Dt/\eps^3$ part of estimates in Theorem~\ref{thm:error_estim}.

\begin{proposition}
\label{prop:E1_y}
Assume that $\bE\in W^{2,\infty}$. {Then, the error of the unique solution of the scheme \eqref{e:dNewton-1st} denoted by $(\bx_\eps^n,\bv_\eps^n,\by_\eps^n)_{n\geq 0}$, in the convergence to  $(\bx_\eps(t_n),\bv_\eps(t_n),\by_\eps(t_n))$ as the exact solution of the system \eqref{e:Newton}, is bounded as}
\begin{align}
\|\bx_\eps^n-\bx_\eps(t_n)\|
&\,\leq\,\|\by_\eps^n-\by_\eps(t_n)\|
+\eps\|\bv_\eps^n-\bv_\eps(t_n)\|\,\\[0.5em]
&\underset{{\scriptscriptstyle {\bE}}}{\lesssim}\,
\frac{\Dt}{\eps^3}\,t_n\,e^{2K_xt_n}
(1+\eps^2)\Big( \|\bv^0_\eps\|+\eps\,(1+t_{n})\Big)\,.
\nonumber
\end{align}
\end{proposition}

\begin{remark}
As implicit in the foregoing statement, due to the fact that nonlinear terms only depend on $(\by,\bv)$ through $\bx=\by+\eps\bv^\perp$, it is expedient to use the norm $\|\cdot\|_{\eps}$ defined by
\begin{equation}\label{d:esp-norm}
\|(\by,\bv)\|_{\eps}:=\|\by\|+\eps\|\bv\|\,.
\end{equation} 
\end{remark}

\begin{proof}
For the sake of conciseness, we {denote numerical errors by}
\[
\ber_{\by}^n:=\by_\eps^n-\by_\eps(t_n)\,,\qquad
\ber_{\bv}^n:=\bv_\eps^n-\bv_\eps(t_n)\,.
\]

By reformulating the velocity update in \eqref{e:dNewton-1st} as 
\begin{align*}
\bv_\eps^{n+1}&=\,\ds\left(\Id+\frac{\Dt}{\eps^2}\J\right)^{-1}\left(\bv_\eps^n
+\frac{\Dt}{\eps}\bE(t_n,\by_\eps^n\,+\,\eps\,(\bv_\eps^n)^\perp)\right),
\end{align*}
and using {the stability of the implicit operator} in Lemma~\ref{l:1st-J}, we obtain the following bounds for $\|\ber_{\by}^{n+1}\|$ and $\|\ber_{\bv}^{n+1}\|$ for all $n\geq 0$:
{\mmn{
\ds\|\ber_{\by}^{n+1}\| \,&\leq\, 
\|\ber_{\by}^{n}\|
\,+\,K_{x}\Dt\,\|(\ber_{\by}^{n},\ber_{\bv}^{n})\|_\eps
\,+\, \Dt\,\|\bftau_{\by}^n\|,\\
\ds\|\ber_{\bv}^{n+1}\|\,&\leq  \Lambda_\lambda^{-1}\left(
\|\ber_{\bv}^n\|\,+\,K_x\dfrac{\Dt}{\eps}\,\|(\ber_{\by}^{n},\ber_{\bv}^{n})\|_\eps
\,+\,\Dt\ds\|\bftau_{\bv}^n\| \right),
}
which, thanks to Lemma~\ref{l:1st-J}, yields,}
\[
\|(\ber_{\by}^{n+1},\ber_{\bv}^{n+1})\|_\eps
\,\leq\,(1+2K_x\Dt)\|(\ber_{\by}^{n},\ber_{\bv}^{n})\|_\eps
+\Dt\,\|(\bftau_{\by}^{n},\bftau_{\bv}^{n})\|_\eps\,,\qquad n\geq0\,.
\]
{Note that $\|(\ber_{\by}^{0},\ber_{\bv}^{0})\|_\eps=0$. Thus, by iteration, we deduce that for any $n>0$,}
\begin{align*}
\|(\ber_{\by}^{n},\ber_{\bv}^{n})\|_\eps
&\ds\leq
\sum_{\ell=0}^{n-1}
(1+2K_x\Dt)^{n-\ell-1}\|(\bftau_{\by}^{\ell},\bftau_{\bv}^{\ell})\|_\eps\ \Dt\,,\\
&\ds\leq
\sum_{\ell=0}^{n-1}
e^{2\,K_x\,(t_{n}-t_{\ell+1})}\|(\bftau_{\by}^{\ell},\bftau_{\bv}^{\ell})\|_\eps\ \Dt\,.
\end{align*}
The proof is, then, concluded by applying {the consistency result} in Lemma~\ref{l:1st_consistency}.
\end{proof}

\subsubsection{$\eps$-independent classical estimate for the asymptotic model} 
Here, we discuss the classical convergence analysis of the asymptotic numerical model \eqref{e:dNgyro-1st} to the guiding center equation \eqref{e:Ngyro}.

\begin{proposition}
\label{prop:E1_gc}
Assume that $\bE\in W^{1,\infty}$. Then, it holds 
\mm{
\|\bx^n-\bx(t_n)\|
&\leq \frac{\Dt}{2}\,t_n\,e^{K_x\,t_n}\,\big(K_t+K_xK_0\big)\,,
}
when $(\bx^{n})_{n\in\N}$ and $\bx(t_n)$ solve respectively \eqref{e:dNgyro-1st} and \eqref{e:Ngyro}, with the same initial {data}.
\end{proposition}
\begin{proof}
{The proof is omitted as, thanks to the $\eps$-independence of the asymptotic model, it boils down to a simpler version of the proof of Proposition~\ref{prop:E1_y}, indeed, the quite classical error analysis of forward Euler integration.}
\end{proof}

\subsection{Asymptotic estimates}
\label{sec:3.3}
To obtain a bound on the asymptotic error $\|\bx^n-\bx_\eps^n\|$, thus a discrete counterpart of Theorem \ref{thm:second_estim}, we refine below the analysis of \cite[Section~4.1]{FR16}.

The first step is to obtain $\eps$-uniform, local-in-time estimates, thus a discrete counterpart of Lemma~\ref{lem:apriori}. To do so, it is convenient to introduce $(\bz^n_\eps)_{n\geq 1}$ as the {discrete analogue of} the auxiliary variable $\bz_\eps$, \ie,
\be\label{def:zn_discrete}
\bz_\eps^n\,:=\,
\bv_\eps^n\,+\,\eps\, \bE^\perp(t_{n-1},\bx_\eps^{n-1}),\qquad\,\, n\,\geq\,1\,.
\ee
Then, from \eqref{e:dNewton-1st}, it follows
\be
\label{eq:zn_discrete}
\ds\frac{\bz_\eps^{n+1}-\bz_\eps^{n}}{\Dt }\,=\,\ds
-\,\frac{1}{\eps^2}(\bz_\eps^{n+1})^\perp
\,+\,\eps\,\frac{\bE^\perp(t_n,\bx_\eps^n)-\bE^\perp(t_{n-1},\bx_\eps^{n-1})}{\Dt},  
\quad{\rm for}\,\, n\geq 1\,.
\ee
Thus, the following lemma provides the required estimates.

\begin{lemma}
\label{lem:disct_zbound}
Assume  that $\bE\in W^{1,\infty}$ and consider \eqref{e:dNewton-1st}. Then, the variable $(\bz^n_\eps)_{n\geq 1}$ defined in \eqref{def:zn_discrete} and the velocity computed by the scheme are bounded, for $n\geq 1$, as
\begin{align*}
\begin{cases}
\ds\|\bz_\eps^n\|
\,\leq\,
e^{K_{x}\,t_n}
\big(\|\bv^0_\eps\|+2\,\eps\,K_0\big)
+\eps\,t_n\,e^{K_{x}\,t_n}\left(K_t \,+\,K_{x}\,K_0\right)\,, 
\\[0.9em]
\ds\|\bv_\eps^n\|\,\leq\,\|\bz_\eps^n\|\,+\,K_0\,\eps\,.
\end{cases}
\end{align*}
\end{lemma}


\begin{proof}
First, we observe that for any $n\geq 1$, 
\mmn{
\|\bE^\perp(t_n,\bx_\eps^n)-\bE^\perp(t_{n-1},\bx_\eps^{n-1})\|
&\leq\,\bigg( K_t\,+\,K_x\frac{\|\bv_\eps^n\|}{\eps}\bigg)\,\Dt\,,\\
&\leq\,\bigg(K_t+K_x\left(\frac{\|\bz_\eps^n\|}{\eps}+K_0\right)\bigg)\,\Dt\,,
} 
which, combined with the $\bz_\eps$-update \eqref{eq:zn_discrete} and Lemma~\ref{l:1st-J}, yields
\begin{align*}
\|\bz_\eps^{n+1}\| &\leq\,\left(1+K_x\,\Dt\right)\, \|\bz_\eps^{n}\| \,+\,
  \eps\,\Dt\, \left( K_t \,+\, K_x\, K_0 \right)\,, \qquad \text{ for } n\geq 1.
\end{align*}
Iterating on the foregoing estimate implies
\begin{align}
\label{eq:zbound_dis_1st}
\|\bz_\eps^{n}\| &\leq\,e^{K_x(t_n-t_1)}\,\|\bz_\eps^{1}\| \,+\,
\eps\,(t_{n}-t_1)e^{K_x(t_n-t_1)}\left( K_t \,+\, K_x\, K_0 \right)\,, \qquad \text{ for } n\geq1\,,
\end{align}
which gives the final estimate, since $\|\bz_\eps^{1}\|\leq \|\bv^1_\eps\|+\eps\,K_0$ and $\bv^1_\eps$ is bounded from \eqref{e:dNewton-1st} and Lemma~\ref{l:1st-J} as
\[
\|\bv^1_\eps\|\leq \|\bv^0_\eps\|\,+\,\eps\,K_0\,.
\]
\end{proof}

With $\eps$-uniform bounds in hands, we are in a position to prove asymptotic estimates.

\begin{proposition} 
\label{prop:disct_err}
\begin{enumerate}[(i)]
\item
Assume $\bE\in W^{1,\infty}$. The difference between flows of \eqref{e:dNewton-1st} and \eqref{e:dNgyro-1st} satisfies, for any $\eps>0$, 
\begin{align*}
\|\bX_\eps^{\Dt}(t_n,0,\bx_\eps^0,\bv_\eps^0)-\bX^{\Dt}(t_n,0,\bx_\eps^0)\|
&\ds\underset{{\scriptscriptstyle {\bE}}}{\lesssim}
\eps\,e^{K_xt_n}\,(1+\,t_n^2)\,\left(\|\bv_\eps^0\|+\eps\right)\,.
\end{align*}
\item 
Assume $\bE\in W^{2,\infty}$. Then, for any $\eps>0$, 
\begin{align*}
\|\bY_\eps^{\Dt}(t_n,0,\bx_\eps^0,\bv_\eps^0)
\,-\,
\bX^{\Dt}(t_n,0,\bx_\eps^0-\eps(\bv_\eps^0)^\perp)\|
\ds\underset{{\scriptscriptstyle {\bE}}}{\lesssim}
e^{2K_xt_n}\,(1+t_n^4)
\,\big(\eps\,\Dt\,\|\bv_\eps^0\|+\eps^2\,(1+\|\bv_\eps^0\|^2+\eps^2)\big)\,.
\end{align*}
\end{enumerate}
\end{proposition}
\begin{proof}
To obtain the {first estimate, we subtract \eqref{e:dNgyro-1st} from \eqref{eq:dy_time} and make a summation from $0$ to $(n-1)$ to get}
\begin{align*}
\by_\eps^n-\bx^n&=\by_\eps^0-\bx^0
\,-\,\Dt\sum_{\ell=0}^{n-1}(\bE^\perp(t_\ell,\bx_\eps^\ell)-\bE^\perp(t_\ell,\bx^\ell))\,, \qquad \text{ for } n\geq0,
\end{align*}
with $(\bx^n)_{n\geq0}:=(\bX^{\Dt}(t_n,0,\bx_\eps^0))_{n\geq0}$. Thus,
\begin{align*}
\|\bx_\eps^{n}-\bx^{n}\|
\,\leq\,
\eps\,(\|\bv_\eps^n\|+\|\bv_\eps^0\|)
+K_x\Dt\sum_{\ell=0}^{n-1}\|\bx_\eps^{\ell}-\bx^{\ell}\|\,,\qquad \text{ for }n\geq0\,.
\end{align*}

Now we use the discrete Gr\"onwall lemma, in the form that
\begin{subequations}
\mm{\label{eq:dgronwall1}
A_n\leq a_n+K\Dt\sum_{\ell=0}^{n-1}A_\ell, \qquad \text{ for } n\geq0,
}
($A$, $a$ and $K$ being non-negative) implies 
\mm{\label{eq:dgronwall2}
 A_n&\leq a_n+K\Dt\sum_{\ell=0}^{n-1} (1+K\Dt)^{n-(\ell+1)} a_\ell
\leq a_n+K\Dt\sum_{\ell=0}^{n-1} e^{K(t_{n}-t_{\ell+1})} a_\ell\,, \qquad \text{ for } n\geq0.
}
\end{subequations}
{Following similar steps as of the continuous case, this leads to the  estimate
\begin{align*}
\|\bx_\eps^{n}-\bx^{n}\|
&\leq\, \eps\,e^{K_xt_n}\,(1+K_x\,t_n)\,
\left(\|\bv_\eps^{0}\|
+\max_{0\leq \ell\leq n}(e^{-K_x t_{\ell}}\|\bv_\eps^{\ell}\|)\right),
&n\geq 0\,,
\end{align*}
which proves the first inequality, thanks to the velocity estimate in Lemma~\ref{lem:disct_zbound}.}

As in continuous case, the proof of the second estimate requires significantly more algebraic manipulations. One starts with the Taylor expansion of the right hand side of \eqref{eq:dy_time}, that is, 
\begin{align}
\label{eq:dy_time-bis}
\frac{\by_\eps^{n+1}-\by_\eps^n}{\Dt}&=-\,\bE^\perp(t_n,\by_\eps^n)\,-\,\eps\,\dD_x\bE^\perp(t_n,\by_\eps^n)(\bv_\eps^{n})^\perp
+\eps^2\Theta_\eps(t_n,\by_\eps^n,\bv_\eps^n)\,,
&n\geq 0\,,
\end{align}
with {the remainder term} $\Theta_\eps$ such that $\|\Theta_\eps(t,\by,\bv)\|\leq \tfrac12 K_{xx}\|\bv\|^2$. So as to rewrite the linear term  $\bv_\eps^n$ in \eqref{eq:dy_time-bis} using again \eqref{e:dNewton-1st}, we observe that, for $n\geq1$, 
\begin{alignat*}{2}
\dD_x\bE(t_n,\by_\eps^n)(\bv_\eps^{n})^\perp
&\,=&&\,-\eps^2\,\dD_x\bE(t_n,\by_\eps^n)\left(\frac{\bv_\eps^{n}-\bv_\eps^{n-1}}{\Dt}\,-\,\frac{1}{\eps}\bE(t_{n-1},\bx_\eps^{n-1})\right)\,,
\\
&\,=&&\,-\frac{\eps^2}{\Dt}\left(\dD_x\bE(t_{n},\by_\eps^{n})\bv_\eps^{n}-\dD_x\bE(t_{n-1},\by_\eps^{n-1})\bv_\eps^{n-1}\right)
+\eps\,\dD_{\bx}\bE(t_n,\by_\eps^n)\,\bE(t_{n-1},\bx_\eps^{n-1})
\\
&\,&&+\,\frac{\eps^2}{\Dt}\left(\dD_x\bE(t_{n},\by_\eps^{n})-\dD_x\bE(t_{n-1},\by_\eps^{n-1})\right)\,\bv_\eps^{n-1}\,,
\end{alignat*}
whose last term can be bounded, using \eqref{eq:dy_time}, as 
\[
\|\dD_x\bE(t_{n},\by_\eps^{n})-\dD_x\bE(t_{n-1},\by_\eps^{n-1})\|
\leq \Dt\,\left(K_{tx}+K_{xx}\,K_0\right)\,, \qquad \text{ for } n\geq1.
\]
{Then, for all $n\geq 0$, we denote $\by^n:=\bX^{\Dt}(t_n,0,\bx_\eps^0-\eps(\bv_\eps^0)^\perp)$, and subtract \eqref{eq:dy_time-bis} from the equation satisfied by $(\by^n)_{n\geq0}$, given by \eqref{e:dNgyro-1st}. Inserting {the latter reformulation} in \eqref{eq:dy_time-bis} and summing give}
\[
\|\by_\eps^n-\by^n\|
\,\leq\,\|\by_\eps^1-\by^1\|+\eps^2r_n\,+\,K_x\Dt\,\sum_{\ell=1}^{n-1}\|\by_\eps^\ell-\by^\ell\|\,, \qquad \text{ for } n\geq 1,
\]
with $r_n$ defined as
\begin{align*}
r_n:= \,\,&
\eps\,K_x\,(\|\bv_\eps^{n-1}\|+\|\bv_\eps^{1}\|)
\,+\,K_x\,K_0\,(t^{n}-t^1)\\
&\,+\,\eps\,\Dt\,(K_{tx}+K_{xx}\,K_0)\sum_{\ell=1}^{n-1}\|\bv_\eps^{\ell-1}\|
\,+\,\Dt\,\frac{K_{xx}}{2}\sum_{\ell=1}^{n-1}\|\bv_\eps^{\ell}\|^2\,.
\end{align*}
{So, the discrete Gr\"onwall lemma yields} 
\begin{align*}
\|\by_\eps^{n}-\by^{n}\|
&\leq\, e^{K_x(t_n-t_1)}\,(1+K_x\,(t_n-t_1))\,
\left(\|\by_\eps^{1}-\by^{1}\|
+\eps^2\,{\max_{2\leq \ell\leq n}(e^{-K_x(t_{\ell}-t_1)}r_\ell)}\right),
&n\geq 1\,.
\end{align*}
{Moreover, the $\by_\eps$-update \eqref{eq:dy_time} implies that, for the first step,}
\[
\|\by_\eps^{1}-\by^{1}\|\,\leq\,K_x\Dt\,\eps\,\|\bv_\eps^0\|\,.
\]
{The proof is, then, concluded by combining this estimate with the bound for the velocity in Lemma~\ref{lem:disct_zbound}.}
\end{proof}

\subsection{Proofs of Theorem~\ref{thm:error_estim} and Corollary~\ref{cor:density_error-estim}}
\label{sec:3.4}

{Lemma \ref{l:1st-J} confirms the unique solvability of the scheme,
  \ie, the matrix $\Id+\lambda \J$ is invertible, so the implicit part
  provides a unique velocity update. As mentioned hereinbefore, the
  error estimates in Theorem~\ref{thm:error_estim} is obtained
  by taking the minimum of the estimates from
  Proposition~\ref{prop:E1_y} and the sum of estimates in
  Theorem~\ref{thm:second_estim} and Propositions~\ref{prop:E1_gc}
  and~\ref{prop:disct_err}. Indeed since  $\eps\,\Dt\lesssim
  \eps^2+\Dt^2 \leq \eps^2 + t_n\Dt $ for $n\geq 1$,  the presence in Proposition~\ref{prop:disct_err} of
  an $\cO(\eps\,\Dt)$-term beside the expected
  $\cO(\eps^2)$-term does not deteriorate the final error estimate.

Corollary~\ref{cor:density_error-estim} is then deduced from Theorem~\ref{thm:error_estim} exactly as Corollary~\ref{cor:density_estim} {was concluded} from Theorem~\ref{thm:second_estim}. For a more abstract version, see \eg{}
\cite[Proposition~2.1]{FR18}.

\begin{remark}
 {The  $\cO(\eps\,\Dt)$-term in Proposition~\ref{prop:disct_err} may be tracked down to the initial step of the scheme \eqref{e:dNewton-1st}.} As already pointed out in \cite[Remark 2.6]{FR17}, similar issues in the asymptotic error rates occur, in general, for higher-order schemes and they do impact the numerical convergence rates. Hence the need to understand how to fix this --- here harmless --- issue. To some extent, the problem may be cured by modifying the initial step of the scheme, \eg, with either a small time step or a fully-implicit treatment. An inspection of the proof of Proposition~\ref{prop:disct_err} shows that, here, setting the initial time step of size $(\Dt)_0\lesssim \eps$ would be sufficient {and resolves the issue}. The analysis of the effect of an implicit treatment requires more work and is discussed in the next section.
\end{remark}

\subsection{Variant with a fully-implicit first step}
{In this section, we discuss the gain in the $\by$-estimate of Proposition~\ref{prop:disct_err} obtained by modifying the first step of the scheme \eqref{e:dNewton-1st} into a fully-implicit version.} We consider the scheme obtained by combining
\mm{
 \label{e:dNewton-1st_stepone}
\begin{cases}
\ds\frac{\bx_\eps^{1}-\bx_\eps^{0}}{\Dt }&=\,\ds\frac{\bv_\eps^{1}}{\eps},
\\[1em]
\ds\eps\,\frac{\bv_\eps^{1}-\bv_\eps^{0}}{\Dt }&=\,\ds
 \bE(t_1,\bx_\eps^1)\,-\,\frac{(\bv_\eps^{1})^\perp}{\varepsilon}.
\end{cases}
}
with \eqref{e:dNewton-1st} for $n\geq1$. Accordingly, the expected asymptotic scheme is the combination of one implicit Euler step followed by explicit Euler steps, {that is,}
\be
\label{e:dNgyro-1st_stepone}
\begin{cases}
\dfrac{\bx^{1}-\bx^{0}}{\Dt }&=\,-\bE^\perp(t_1,\bx^1)\,, 
\\[1.1em]
\dfrac{\bx^{n+1}-\bx^{n}}{\Dt }&=\,-\bE^\perp(t_n,\bx^n)\,, \qquad {\rm
  for }\,\, n\geq 1\,.
\end{cases}
\ee

A rash inspection may lead to the deceptive conclusion that a time
step $\Dt=\cO(\eps)$ is required to solve
\eqref{e:dNewton-1st_stepone} (for instance applying a Newton's method). However, this first step \eqref{e:dNewton-1st_stepone} can be equivalently written in terms of $(\by,\bv)$ as
 {\[(\by_\eps^1,\bv_\eps^1)=\bF(\by_\eps^1,\bv_\eps^1),\]
where $\bF=(\bF_{\by},\bF_{\bv})$ defined as
\mmn{
\begin{cases}
\ds\bF_{\by}(\by,\bv)&=\,\ds\by_\eps^0-\Dt\,\bE^\perp(t_1,\by\,+\,\eps\,(\bv)^\perp)\,,
\\[1em]
\ds\bF_{\bv}(\by,\bv)&=\,\ds\left(\Id+\frac{\Dt}{\eps^2}\J\right)^{-1}\left(\bv_\eps^0
+\frac{\Dt}{\eps}\bE(t_1,\by\,+\,\eps\,(\bv)^\perp)\right).
\end{cases}
}
So, it can be seen that the map $\bF$ has a Lipschitz constant which is not bigger than $2K_x\Dt$ for the norm $\|\cdot\|_{\eps}$ introduced in \eqref{d:esp-norm}. Therefore, thanks to the contraction mapping theorem, the unique solvability of the first step \eqref{e:dNewton-1st_stepone} is assured for a small enough time step, explicitly when $K_x\Dt<1/2$. For this alternate scheme, we prove the following modification of the second asymptotic estimate of Proposition \ref{prop:disct_err}.}

\begin{proposition} 
\label{prop:disct_err_stepone}
Assume $\bE\in W^{2,\infty}$ and $K_x\Dt<1/2$. Then, 
\begin{align*}
\|\bY_\eps^{\Dt}(t_n,0,\bx_\eps^0,\bv_\eps^0)
\,-\,
\bX^{\Dt}(t_n,0,\bx_\eps^0-\eps(\bv_\eps^0)^\perp)\|\,\ds\underset{{\scriptscriptstyle {\bE}}}{\lesssim}
\,\eps^2\,\frac{e^{2K_xt_n}\,(1+t_n^4)}{1-K_x\Dt}
\,(1+\|\bv_\eps^0\|^2+\eps^2)\,,
\end{align*}
{where $\bY_\eps^{\Dt}:=\bX_\eps^{\Dt}-\eps(\bV_\eps^{\Dt})^\perp$, and $\bX^{\Dt}$ and $(\bX_\eps^{\Dt},\bV_\eps^{\Dt})$ denote, respectively, the discrete asymptotic flow \eqref{e:dNgyro-1st_stepone} and the discrete flow obtained from \eqref{e:dNewton-1st_stepone}--\eqref{e:dNewton-1st}.}
\end{proposition}

\begin{remark}
{Note that in Proposition \ref{prop:disct_err_stepone}, one could, instead, impose the constraint $K_x(\Dt)_0<1/2$ only on the first step. }
\end{remark}

\begin{proof}
As in Proposition \ref{prop:disct_err}, an $\eps$-uniform bound on the computed velocity is a prerequisite for the asymptotic estimate. The proof of such a bound is identical to the one of Lemma~\ref{lem:disct_zbound} except for the first step, where we need to re-define $\bz^1_\eps$ as
\mmn{
\bz_\eps^1&:=\bv_\eps^1\,+\,\eps\, \bE^\perp(t_{1},\bx_\eps^{1}).
}
Indeed, with this modification \eqref{eq:zn_discrete} is unaltered and, despite the modification of the first step, Lemma~\ref{l:1st-J} still implies
\[
\|\bv^1_\eps\|\leq \|\bv^0_\eps\|\,+\,\eps\,K_0\,.
\] 
so that the conclusion of Lemma~\ref{lem:disct_zbound} still holds for the new scheme.

Thus, as in the proof of Proposition \ref{prop:disct_err} and with the same $r_n$ thereof, one gets {the following estimate for all  $n\geq 1$ 
\begin{align*}
\|\by_\eps^{n}-\by^{n}\|
&\leq\, e^{K_x(t_n-t_1)}\,(1+K_x\,(t_n-t_1))\,
\left(\|\by_\eps^{1}-\by^{1}\|
+\eps^2\,\max_{2\leq \ell\leq n}(e^{-K_x(t_{\ell}-t_1)}r_\ell)\right).
\end{align*}
}

{As for the first step, with the same $\Theta_\eps$ {as before},} 
\[
\frac{\by_\eps^{1}-\by_\eps^0}{\Dt}=-\,\bE^\perp(t_1,\by_\eps^1)\,-\,\eps\,\dD_x\bE^\perp(t_1,\by_\eps^1)(\bv_\eps^{1})^\perp+\eps^2\Theta_\eps(t_1,\by_\eps^1,\bv_\eps^1)\,,
\]
{The reformulation of $\dD_x\bE(t_1,\by_\eps^1)(\bv_\eps^{1})^\perp$ writes}
\mmn{
\dD_x\bE(t_1,\by_\eps^1)(\bv_\eps^{1})^\perp
\,=\,&-\eps^2\,\dD_x\bE(t_1,\by_\eps^1)\left(\frac{\bv_\eps^{1}-\bv_\eps^{0}}{\Dt}\,-\,\frac{1}{\eps}\bE(t_{1},\bx_\eps^{1})\right)\,
\\
\,=\,&-\frac{\eps^2}{\Dt}\left(\dD_x\bE(t_{1},\by_\eps^{1})\bv_\eps^{1}-\dD_x\bE(t_{0},\by_\eps^{0})\bv_\eps^{0}\right)
+\eps\,\dD_{\bx}\bE(t_1,\by_\eps^1)\,\bE(t_{1},\bx_\eps^{1})
\\
\,&+\,\frac{\eps^2}{\Dt}\left(\dD_x\bE(t_{1},\by_\eps^{1})-\dD_x\bE(t_{0},\by_\eps^{0})\right)\,\bv_\eps^{0}\,,
}
which yields the estimate
\begin{align*}
\|\by_\eps^{1}-\by^{1}\|
\,\leq\,\frac{\eps^2}{1-K_x\Dt}\Bigg(&
\eps\,K_x\,(\|\bv_\eps^1\|+\|\bv_\eps^0\|)
\,+\,\Dt\,K_x\,K_0\\
&\,+\,\eps\,\Dt\,(K_{tx}+K_{xx}\,K_0)\,\|\bv_\eps^0\|
\,+\,\Dt\,\frac{K_{xx}}{2}\,\|\bv_\eps^1\|^2\Bigg)\,.
\end{align*}
From here the proof is completed as the one in Proposition~\ref{prop:disct_err} by collecting all the foregoing estimates.
\end{proof}

{It is essential to emphasize that numerical schemes relying on a distinct treatment of the initial step are not really desirable as they are not well-adapted to the case when the magnetic field is not uniformly strong. Namely, such schemes are inefficient when effectively the scaling parameter starts of size $\eps\sim 1$ and later gets small, $\eps\ll 1$. For this reason, we will consider a different type of remedy.}

\section{First-order scheme on the guiding center variable}
\label{sec:4}
We propose in this section a modification of the scheme \eqref{e:dNewton-1st}, which is simply
based on the approximation of $(\by_\eps,\bv_\eps)$ instead of $(\bx_\eps,\bv_\eps)$. We define a discrete time evolution by a first-order IMEX method for $n\geq 0$,
\mm{\label{e:dNewton-1st_y}
\begin{cases}\ds
\frac{\bv_\eps^{n+1}-\bv_\eps^{n}}{\Dt }&=\,\ds
\frac1\eps\,\bE(t_n,\by_\eps^{n}+\eps(\bv_\eps^{n})^\perp))\,-\,\frac{1}{\varepsilon^2}(\bv_\eps^{n+1})^\perp,
\\[1em]\ds 
\frac{\by_\eps^{n+1}-\by_\eps^{n}}{\Dt }&=\,-\bE^\perp\left(t_n,\by_\eps^n+\eps(\bv_\eps^{n+1})^\perp\right).
\end{cases}
}

The aim is {the improvement} of the asymptotic error, $\by_\eps^n-\by^n$ where $(\by^n)_n$ solves \eqref{e:dNgyro-1st}, the explicit Euler scheme for the guiding center equation \eqref{e:Ngyro}, with initial datum $\by^0=\by_\eps^0$. Note that {the $\by_\eps$-update in \eqref{e:dNewton-1st_y}  differs from \eqref{eq:dy_time}}, only by the presence of $\bv_\eps^{n+1}$ instead of $\bv_\eps^{n}$. {We will see (in Proposition \ref{prop:disct_err_y} below) that this simple difference will remove the unwanted term of Proposition~\ref{prop:disct_err}.}

In terms of the spatial variable $\bx_\eps^n:=\by_\eps^n+\eps(\bv_\eps^n)^\perp$, for $n\geq 0$, the scheme \eqref{e:dNewton-1st_y} is equivalently written as 
\mm{\label{e:dNewton-1st_y_xintime}
\begin{cases}\ds
\frac{\bv_\eps^{n+1}-\bv_\eps^{n}}{\Dt }&=\,\ds
\frac1\eps\,\bE(t_n,\bx_\eps^{n})\,-\,\frac{1}{\varepsilon^2}(\bv_\eps^{n+1})^\perp,
\\[1em]\ds 
\frac{\bx_\eps^{n+1}-\bx_\eps^{n}}{\Dt }&=\,\ds
\frac{\bv_\eps^{n+1}}{\eps}+\left[\bE^\perp(t_n,\bx_\eps^n)-\bE^\perp\left(t_n,\bx_\eps^n+\eps\left(\bv_\eps^{n+1}-\bv_\eps^{n}\right)^\perp\right)\right].
\end{cases}
}
Compared to \eqref{e:dNewton-1st}, the modification to $(\bx_\eps^{n})_{n\geq1}$ is expected to be $\cO(\eps\,\Dt)$, hence immaterial for asymptotic and numerical convergences of the variable $\bx_\eps^n$. {The following theorem provides the main error estimate concerning the modified scheme \eqref{e:dNewton-1st_y}.}

\begin{theorem}
\label{thm:error_estim_1st_y}
{The first-order scheme \eqref{e:dNewton-1st_y} possesses a unique solution. Moreover}	
 \begin{enumerate}[(i)]
 \item when $\bE\in W^{1,\infty}$, the space variable $\bx_\eps$ satisfies for all $n\geq0$, $\Dt>0$ and $\eps>0$,
\mmn{
\|\bx_\eps^n-\bx_\eps(t_n)\|
&\underset{{\scriptscriptstyle {\bE}}}{\lesssim}\,
(1+t_{n}^3)\,e^{(2+K_x\Dt)K_xt_n}(1+\|\bv_\eps^0\|+\eps)
\times\min\Bigg(\frac{\Dt}{\eps^3}(1+\eps^2),
\Dt+\eps\Bigg),
}
\item when $\bE\in W^{2,\infty}$, the guiding center variable $\by_\eps$ satisfies for all $n\geq0$, $\Dt>0$ and $\eps>0$,
\mmn{
\|\by_\eps^n-\by_\eps(t_n)\|
&\underset{{\scriptscriptstyle {\bE}}}{\lesssim}\,
(1+t_{n}^4)\,e^{(2+K_x\Dt)K_xt_n}(1+\|\bv_\eps^0\|^2+\eps^2)
\times\min\Bigg(\frac{\Dt}{\eps^3}\,(1+\eps)\,,\Dt+\eps^2\,\Bigg)\,.
}
\end{enumerate}
\end{theorem}

{Our strategy to prove Theorem \ref{thm:error_estim_1st_y} follows closely the one to prove Theorem \ref{thm:error_estim} in \S\ref{sec:3}.}

\subsection{Direct convergence estimates}
\label{sec:4.2}

To perform a direct numerical convergence analysis, we, first, introduce corresponding consistency errors,
\mm{
\label{eq:tauv_1st_y}
\bftau_{\by}^n&:=
\ds\frac{\by_\eps(t_{n+1})-\by_\eps(t_{n})}{\Dt}\,+\,\bE^\perp(t_n,\by_\eps(t_n)\,+\,\eps\,\bv_\eps^\perp(t_{n+1}))\,,\\
\label{eq:tauy_1st_y}
\bftau_{\bv}^n&:=
\ds\frac{\bv_\eps(t_{n+1})-\bv_\eps(t_{n})}{\Dt}\,-\,\dfrac{\bE(t_n,\by_\eps(t_n)+\eps\, \bv_\eps^\perp(t_n))}{\eps}\,+\,\dfrac{\bv_\eps^\perp(t_{n+1})}{\eps^2}\,,
}
{which can be estimated as in Lemma~\ref{l:1st_consistency}. Note that $\bftau_{\bv}^n$ in \eqref{eq:tauv_1st_y} is identical to \eqref{eq:tauv_1st} whereas $\bftau_{\by}^n$ differs from
  \eqref{eq:tauy_1st} by the presence of $\bv_\eps^{n+1}$ instead of $\bv_\eps^n$. The corresponding modification to the proof of Lemma~\ref{l:1st_consistency} leads to a statement essentially identical to Lemma~\ref{l:1st_consistency}, hence omitted here. From this we derive the following proposition.}

\begin{proposition}
\label{prop:E1_y_y}
Assume that $\bE\in W^{2,\infty}$. {The error from the unique solution of the scheme \eqref{e:dNewton-1st_y} to the exact solution of the system \eqref{e:Newton} is bounded as}
\begin{align}
\|\bx_\eps^n-\bx_\eps(t_n)\|
&\,\leq\,\|\by_\eps^n-\by_\eps(t_n)\|
+\eps\|\bv_\eps^n-\bv_\eps(t_n)\|\,\\[0.5em]
&\underset{{\scriptscriptstyle {\bE}}}{\lesssim}\,
\frac{\Dt}{\eps^3}\,t_n\,e^{2K_xt_n}
(1+\eps^2)(1+\Dt)\Big( \|\bv^0_\eps\|+\eps\,(1+t_{n})\Big)\,.
\nonumber
\end{align}
\end{proposition}

\begin{proof}
As in the proof of Proposition~\ref{prop:E1_y}, we consider {numerical errors}
\[
\ber_{\by}^n:=\by_\eps^n-\by_\eps(t_n)\,,\qquad\qquad 
\ber_{\bv}^n:=\bv_\eps^n-\bv_\eps(t_n)\,,
\]
and use the norm $\|\cdot\|_\eps$ introduced in \eqref{d:esp-norm}. By reformulating the second equation of \eqref{e:dNewton-1st_y} and using Lemma~\ref{l:1st-J}, we obtain
\begin{align*}\ds
\|\ber_{\by}^{n+1}\|
&\ds
\leq\,\|\ber_{\by}^{n}\|+K_x\Dt\,(\|\ber_{\by}^{n}\|+\eps\|\ber_{\bv}^{n+1})\|)
+\Dt\,\|\bftau_{\by}^{n}\|\,,\\[0.5em]
\eps\|\ber_{\bv}^{n+1}\|
&\ds
\leq\,\eps\|\ber_{\bv}^{n}\|+K_x\Dt\,\|(\ber_{\by}^{n},\ber_{\bv}^{n})\|_\eps
+\Dt\,\eps\,\|\bftau_{\bv}^{n}\|\,,
\end{align*}
{that may be combined to yield
\[
\|(\ber_{\by}^{n+1},\ber_{\bv}^{n+1})\|_\eps
\,\leq\,(1+K_x\Dt)^2\|(\ber_{\by}^{n},\ber_{\bv}^{n})\|_\eps
+\Dt\,(1+K_x\Dt)\,\|(\bftau_{\by}^{n},\bftau_{\bv}^{n})\|_\eps\,.
\]}
Iterating and applying our new version of Lemma~\ref{l:1st_consistency} {conclude} the proof. Let us observe that when doing so we use a slightly different form of the discrete Gr\"onwall where $(1+K_x\Dt)^2$ leads to the same exponential factor as $(1+2K_x\Dt)$ was, both being bounded by $e^{2K_x\Dt}$.
\end{proof}

\subsection{Asymptotic estimates}
\label{sec:4.3}

To prove asymptotic estimates, we first obtain uniform bounds on solutions to \eqref{e:dNewton-1st_y}. To do so, as in \S\ref{sec:3.3}, we consider a discrete analogue of the auxiliary variable $\bz_\eps$, denoted by $(\bz^n_\eps)_{n\geq 0}$, defined exactly as in \eqref{def:zn_discrete} and also satisfying \eqref{eq:zn_discrete}.

\begin{lemma}
\label{lem:disct_zbound_y}
{Assume  that $\bE\in W^{1,\infty}$ and consider \eqref{e:dNewton-1st_y}. Then, the variable $(\bz^n_\eps)_{n\geq 0}$ defined in \eqref{def:zn_discrete} and the velocity computed by the scheme are bounded, for $n\geq 1$, as}
\begin{align*}
\begin{cases}
\ds\|\bz_\eps^n\|
\,\leq\,
e^{K_{x}\,t_n}
\big(\|\bv^0_\eps\|+2\eps\,K_0\big)
+\eps\,t_n\,e^{K_{x}\,t_n}\left(K_t \,+\,K_{x}\,K_0\right)\,, 
\\[0.9em]
\ds\|\bv_\eps^n\|\,\leq\,\|\bz_\eps^n\|\,+\,K_0\,\eps\,.
\end{cases}
\end{align*}
\end{lemma}

\begin{proof}
The proof goes along the same lines as the proof of Lemma~\ref{lem:disct_zbound}. The only difference is that now one has directly, for $n\geq1$,
\mmn{
\|\bE^\perp(t_n,\bx_\eps^n)-\bE^\perp(t_{n-1},\bx_\eps^{n-1})\|
&\leq\,\left(K_t+K_x\left(\frac{\|\bz_\eps^n\|}{\eps}+K_0\right)\right)\,\Dt\,,
} 
from the $\bx_\eps$-update
\[
\frac{\bx_\eps^{n}-\bx_\eps^{n-1}}{\Dt }=\,\ds
\frac{\bz_\eps^{n}}{\eps}-\bE^\perp\left(t_{n-1},\bx_\eps^{n-1}+\eps\left(\bv_\eps^{n}-\bv_\eps^{n-1}\right)^\perp\right)\,.
\]
\end{proof}

To state asymptotic estimates, we denote the discrete flow for the scheme \eqref{e:dNewton-1st_y} by $(\bX_\eps^{\Dt},\bV_\eps^{\Dt})$, and the one for \eqref{e:dNgyro-1st} by $\bX^{\Dt}$.  {As before we also define $\bY_\eps^{\Dt}:=\bX_\eps^{\Dt}-\eps(\bV_\eps^{\Dt})^\perp$.} 

\begin{proposition} 
\label{prop:disct_err_y}
\begin{enumerate}[(i)]
\item
Assume $\bE\in W^{1,\infty}$. Then, 
\begin{align*}
\|\bX_\eps^{\Dt}(t_n,0,\bx_\eps^0,\bv_\eps^0)-\bX^{\Dt}(t_n,0,\bx_\eps^0)\|
&\ds\underset{{\scriptscriptstyle {\bE}}}{\lesssim}
\eps\,e^{K_xt_n}\,(1+\,t_n^3)\,\left(\|\bv_\eps^0\|+\eps\right)\,.
\end{align*}
\item 
Assume $\bE\in W^{2,\infty}$. Then,
\begin{align*}
\|\bY_\eps^{\Dt}(t_n,0,\bx_\eps^0,\bv_\eps^0)\,-\,
\bX^{\Dt}(t_n,0,\bx_\eps^0-\eps(\bv_\eps^0)^\perp)\|\ds\underset{{\scriptscriptstyle {\bE}}}{\lesssim}
\eps^2\,e^{2K_xt_n}\,(1+t_n^4)
\,\big(1+\|\bv_\eps^0\|^2+\eps^2\big)\,.
\end{align*}
\end{enumerate}
\end{proposition}

\begin{proof}
{To obtain the first estimate, we subtract the $\by_\eps$-update in \eqref{e:dNewton-1st_y} from the limit scheme \eqref{e:dNgyro-1st}, and make a summation from $0$ to $(n-1)$ to obtain, for $n\geq0$,}
\begin{align*}
\by_\eps^n-\bx^n&=\by_\eps^0-\bx^0
\,-\,\Dt\sum_{\ell=0}^{n-1}\left(\bE^\perp\left(t_\ell,\bx_\eps^\ell+\eps\left(\bv_\eps^{\ell+1}-\bv_\eps^{\ell}\right)^\perp\,\right)-\bE^\perp(t_\ell,\bx^\ell)\right)\,,
\end{align*}
{which, thanks to the Lipschitz continuity of the electric field, implies}
\begin{align*}
\|\bx_\eps^{n}-\bx^{n}\|
\,\leq\,
\eps\,(\|\bv_\eps^n\|+\|\bv_\eps^0\|)
+K_x\,\eps\Dt\sum_{\ell=0}^{n-1}(\|\bv_\eps^{\ell}\|+\|\bv_\eps^{\ell+1}\|)
+K_x\Dt\sum_{\ell=0}^{n-1}\|\bx_\eps^{\ell}-\bx^{\ell}\|\,.
\end{align*}
Applying the discrete Gr\"onwall lemma, this leads to 
\begin{align*}
\|\bx_\eps^{n}-\bx^{n}\|
&\underset{{\scriptscriptstyle {\bE}}}{\lesssim}\, \eps\,e^{K_xt_n}\,(1+t_n^2)\,
\left(\|\bv_\eps^{0}\|
+\max_{0\leq \ell\leq n}(e^{-K_x t_{\ell}}\|\bv_\eps^{\ell}\|)\right),
&n\geq 0\,.
\end{align*}
{Finally, applying Lemma~\ref{lem:disct_zbound_y} concludes the proof of the first estimate.}

{The proof of the second estimate starts with a Taylor expansion of the $\by_\eps$-update,}
\begin{align*}
\frac{\by_\eps^{n+1}-\by_\eps^n}{\Dt}&=-\,\bE^\perp(t_n,\by_\eps^n)\,-\,\eps\,\dD_x\bE^\perp(t_n,\by_\eps^n)(\bv_\eps^{n+1})^\perp
+\eps^2\Theta_\eps(t_n,\by_\eps^n,\bv_\eps^{n+1})\,,
&n\geq 0\,,
\end{align*}
with $\Theta_\eps$ such that $\|\Theta_\eps(t,\by,\bv)\|\leq \tfrac12 K_{xx}\|\bv\|^2$. Now, so as to rewrite the linear term, we observe that, when $n\geq0$,    
\mmn{
\dD_x\bE(t_n,\by_\eps^n)(\bv_\eps^{n+1})^\perp
\,=\,&-\eps^2\,\dD_x\bE(t_n,\by_\eps^n)\left(\frac{\bv_\eps^{n+1}-\bv_\eps^{n}}{\Dt}\,-\,\frac{1}{\eps}\bE(t_{n},\bx_\eps^{n})\right)\,
\\
\,=\,&-\frac{\eps^2}{\Dt}\left(\dD_x\bE(t_{n+1},\by_\eps^{n+1})\bv_\eps^{n+1}-\dD_x\bE(t_{n},\by_\eps^{n})\bv_\eps^{n}\right)
+\eps\,\dD_{\bx}\bE(t_n,\by_\eps^n)\,\bE(t_{n},\bx_\eps^{n})
\\
\,&+\,\frac{\eps^2}{\Dt}\left(\dD_x\bE(t_{n+1},\by_\eps^{n+1})-\dD_x\bE(t_{n},\by_\eps^{n})\right)\,\bv_\eps^{n},
}
{whose last term is bounded as}
\[
\|\dD_x\bE(t_{n+1},\by_\eps^{n+1})-\dD_x\bE(t_{n},\by_\eps^{n})\|
\leq \Dt\,\left(K_{tx}+K_{xx}\,K_0\right)\,.
\]
Then, with $\by^n:=\bX^{\Dt}(t_n,0,\bx_\eps^0-\eps(\bv_\eps^0)^\perp)$ for $n\geq0$, and arguing as in the proof of Proposition~\ref{prop:disct_err}, we deduce
$$
\|\by_\eps^{n}-\by^{n}\|
\leq\,\eps^2\,e^{K_xt_n}\,(1+K_x\,t_n)\,
\max_{1\leq \ell\leq n}(e^{-K_xt_{\ell}}r_\ell)\,,\qquad
n\geq 0\,,
$$
where, for $n\geq0$,
\begin{align*}
r_n:=\,\,&
\eps\,K_x\,(\|\bv_\eps^{n}\|+\|\bv_\eps^{0}\|)
\,+\,K_x\,K_0\,t^{n}+\eps\,\Dt\,(K_{tx}+K_{xx}\,K_0)\sum_{\ell=0}^{n-1}\|\bv_\eps^{\ell-1}\|
\,+\,\Dt\,\frac{K_{xx}}{2}\sum_{\ell=1}^{n}\|\bv_\eps^{\ell}\|^2\,.
\end{align*}
The proof is concluded with Lemma~\ref{lem:disct_zbound_y}.
\end{proof}

\subsection{Proof of Theorem~\ref{thm:error_estim_1st_y}}
\label{sec:4.4}

{The proof of Theorem~\ref{thm:error_estim_1st_y} goes along the same lines as of the proof of Theorem~\ref{thm:error_estim}. It consists, {on the one hand}, in applying  Proposition~\ref{prop:E1_y_y} and, on the other hand, in combining Theorem~\ref{thm:second_estim} and Propositions~\ref{prop:E1_gc} and~\ref{prop:disct_err_y}.}

\section{L-stable second-order implicit-explicit scheme}
\label{sec:5}
\setcounter{equation}{0}
In this section, we discuss {the uniform convergence analysis of a} second-order extension of the scheme presented in \S\ref{sec:4}. The analysis we are going to perform is conceptually similar to the ones in \S\ref{sec:3} and \S\ref{sec:4} but technically much more involved.

For concision's sake we restrict to the analysis of a single scheme, a scheme written on the guiding center variable. However, a similar analysis could be performed on the second-order version of the scheme of \S\ref{sec:3} (\cf \cite{FR16}). We stress that here also the deterioration of asymptotic estimates on the guiding variable $\by$ by an $\cO(\eps\,\Dt)$-term has no impact on numerical convergence errors since $\eps\,\Dt\lesssim\eps^2+\Dt^2$.  
 
The {semi-implicit} second-order method we consider is a combination of a Runge--Kutta method (for the explicit part) and an L-stable second-order SDIRK method (for the implicit part), with the parameter $\gamma$ chosen as the smallest root of the polynomial $\gamma^2-2\gamma+{1}/{2}=0$, that is, $\gamma=1-\frac{1}{\sqrt{2}}$; see \cite{boscarino2016high}. 

{To shed some light} on the structure of the scheme, we write the characteristic system \eqref{e:Newton} in the slightly more abstract form, 
\mm{\label{e:Newton2}
\begin{cases}
\ds\by_\eps'(t)\,=\,\bF_{\by}(t,\by_\eps(t),\eps\,\bv_\eps(t))\,,\\[1em]
\ds
\eps\,\bv_\eps'(t)\,=\,\bF^\eps_{\bv}(t,\by_\eps(t),\eps\,\bv_\eps(t),\eps\,\bv_\eps(t))\,,\\[1em]
\by_\eps(s)\,=\,\by_\eps^s,\qquad \bv_\eps(s)\,=\,\bv_\eps^s,
\end{cases}}
with
\begin{align*}
\ds\bF_{\by}(t,\hy,\tbw)
&:=-\bE^\perp(t,\hy+\tbw^\perp)\,,&
\bF^\eps_{\bv}(t,\hy,\hw,\tbw)&:=\bE(t,\hy+\hw^\perp)\,-\,\frac{\tbw^\perp}{\eps^2}\,,
\end{align*}
where a variable with a tilde $\widetilde{b}$ is used in stiff parts of the system while a hatted variables $\widehat{b}$ are to be used in non-stiff parts. The identification of stiff and non-stiff parts in \eqref{e:Newton2} prepares duplication at discrete level of the velocity variable $\bv$ as $\tbv$ and $\hv$ to be treated, respectively, implicitly and explicitly. Such a duplication is essential to obtain a semi-implicit scheme that avoids nonlinear iterations; see \cite{boscarino2016high}.

{The first stage of the scheme is a linearly-implicit update which provides an approximation of the velocity after a time step of size $\gamma\Dt$,}    
\begin{subequations}
\label{eq:lstab_stg}
\mm{
\label{eq:lstab_stg1}
\begin{cases}
\ds\ttime_{n+1}\,:=\,t_n\,+\,\gamma\Dt\,,\\[1.1em]
\ds\eps\frac{\tbv^{n+1}_\eps-\bv^n_\eps}{\gamma\Dt}\,=\,
\bF^\eps_{\bv}(t_n,\by_\eps^n,\eps\,\bv_\eps^n,\eps\,\tbv_\eps^{n+1})\,,
\end{cases}
}
{where $\tbv_\eps^{n+1}$ approximates $\bv_\eps(\ttime_{n+1})$, and is to be used in stiff parts.}

Then, the second stage provides an explicit approximation $(\hy_\eps^{n+1},\hv_\eps^{n+1})$ of $(\by_\eps,\bv_\eps)(\htime_{n+1})$ to be used in non stiff parts. It reads, for $n\geq0$,
\mm{
\label{eq:lstab_stg2-1}
\begin{cases}
\ds\htime_{n+1}\,:=\,t_n\,+\,\frac{\Dt }{2\gamma}\,,\\[1.1em]
\ds\frac{\hy_\eps^{n+1}-\by^n_\eps}{\Dt/(2\gamma)}\,=\,
\bF_{\by}(t_n,\by_\eps^n,\eps\,\tbv_\eps^{n+1})\,,    
\\[1.1em]
\ds\eps\frac{\hv_\eps^{n+1}-\bv^n_\eps}{\Dt/(2\gamma)}
\,=\,\bF^\eps_{\bv}(t_n,\by_\eps^n,\eps\,\bv_\eps^n,\eps\,\tbv_\eps^{n+1})\,.
\end{cases}
}
The last stage, {which provides the final update}, is linearly-implicit and writes
\mm{
\label{eq:lstab_stg2-0}
\begin{cases}
\ds\frac{\by_\eps^{n+1}-\by^n_\eps}{\Dt}
&=\,\,\ds(1-\gamma)\,\bF_{\by}(t_n,\by_\eps^n,\eps\,\tbv_\eps^{n+1})
+\gamma\,\bF_{\by}(\htime_{n+1},\hy_\eps^{n+1},\eps\,\bv_\eps^{n+1})\,,
\\[0.95em]
\ds\eps\frac{\bv_\eps^{n+1}- \bv^{n}_\eps}{\Dt}
&=\,\,\ds(1-\gamma)\,\bF_{\bv}^\eps(t_n,\by_\eps^n,\eps\,\bv_\eps^{n},\eps\,\tbv_\eps^{n+1})
+\gamma\,\bF_{\bv}^\eps(\htime_{n+1},\hy_\eps^{n+1},\eps\,\hv_\eps^{n+1},\eps\,\bv_\eps^{n+1})\,.
\end{cases}
}
\end{subequations} 

Since $\bF_{\by}(t,\by,0)=-\bE^\perp(t,\by)$, taking formally the limit $\eps\to0$ suggests that the above discretization tends to the discretization of the guiding center equation by a second-order fully explicit Runge--Kutta scheme, where the first stage is 
\begin{subequations}
\mm{
\label{e:dNgyro-2nd_1}
\begin{cases}
\ds\htime_{n+1}\,:=\,t_n\,+\,\frac{\Dt }{2\gamma}\,,\\[1.1em]
\ds\frac{\hx^{n+1}-\bx^n}{\Dt/(2\gamma)}\,=\,-\bE^\perp(t_n,\bx^n)\,,
\end{cases}
}
and the second one reads
\be
\label{e:dNgyro-2nd}    
\ds
\frac{\bfx^{n+1}-\bx^n}{\Dt}
\,=\,-(1-\gamma)\,\bE^\perp(t_n,\bfx^n)
-\gamma\,\bE^\perp(\htime_{n+1},\hx^{n+1})\,.
\ee
\end{subequations}

{On this second-order scheme \eqref{eq:lstab_stg}, our main result is the following theorem.}

\begin{theorem}
\label{thm:error_estim_2nd}
{The second-order scheme \eqref{eq:lstab_stg} possesses a unique solution. Moreover }	
 \begin{enumerate}[(i)]
\item there exists $C_0$ such that when $\bE\in W^{2,\infty}$, the space variable $\bx_\eps$ satisfies\\for all $n\geq0$, $\Dt>0$ and $\eps>0$,
\mmn{
\ds\|\bx_\eps^n-\bx_\eps(t_n)\|\,\underset{{\scriptscriptstyle \bE}}{\lesssim}
e^{C_0\,K_x\,t_n}\,(1+t_n)^3\,(1+\eps^2+\|\bv_\eps^0\|^2)
\,\min\left(\frac{\Dt^2}{\eps^5}(1+\eps^3),\eps+\Dt^2\right),
}
\item there exists $C_0$ such that when $\bE\in W^{2,\infty}$, the guiding center variable $\by_\eps$ satisfies\\for all $n\geq0$, $\Dt>0$ and $\eps>0$,
\mmn{
\ds\|\by_\eps^n-\by_\eps(t_n)\|\,\underset{{\scriptscriptstyle \bE}}{\lesssim}
e^{C_0\,K_x\,t_n}\,(1+t_n)^4\,(1+\eps^2+\|\bv_\eps^0\|^2)
\,\min\left(\frac{\Dt^2}{\eps^5}(1+\eps^3),\eps^2+\Dt^2\right).
}
\end{enumerate}
\end{theorem}

The proof of this theorem follows the same strategy as the one of {Theorems~\ref{thm:error_estim} and~\ref{thm:error_estim_1st_y}.} 

\subsection{Direct convergence estimates}
\label{sec:5.2}

To carry out a direct convergence analysis, we introduce consistency
errors for each stage of the scheme. We define for the first stage \eqref{eq:lstab_stg1}
\mm{
  \label{eq:tau_stg0_2nd}
\ttauv^n\,:=\,
\ds\frac{\bv_\eps(\ttime_{n+1})-\bv_\eps(t_{n})}{\gamma\Dt}\,-\,\dfrac{\bE(t_n,\by_\eps(t_n)+\eps\,
  \bv_\eps^\perp(t_n))}{\eps}\,+\,\dfrac{\bv_\eps^\perp(\ttime_{n+1})}{\eps^2}\,,
}
for the intermediate stage \eqref{eq:lstab_stg2-1}
\mm{\label{eq:tau_stg1_2nd}
\begin{cases}
\htauy^n\;:=&
\ds\frac{\by_\eps(\htime_{n+1})-\by_\eps(t_{n})}{\Dt/(2\gamma)}\,+\,\bE^\perp(t_n,\by_\eps(t_n)\,+\,\eps\,\bv_\eps^\perp(\ttime_{n+1}))\,,
\\[1.1em]
\htauv^n\,:=&
\ds\frac{\bv_\eps(\htime_{n+1})-\bv_\eps(t_{n})}{\Dt/(2\gamma)}\,-\,\dfrac{\bE(t_n,\by_\eps(t_n)+\eps\,
  \bv_\eps^\perp(t_n))}{\eps}\,+\,\dfrac{\bv_\eps^\perp(\ttime_{n+1})}{\eps^2}\,,
\end{cases}
}
and for the final stage \eqref{eq:lstab_stg2-0}
\mm{\label{eq:tau_stg2_2nd}
\begin{cases}
\bftau_{\by}^n\,:=\,&
\ds\frac{\by_\eps(t_{n+1})-\by_\eps(t_{n})}{\Dt}\,+\,(1-\gamma)\bE^\perp(t_n,\by_\eps(t_n)\,+\,\eps\,\bv_\eps^\perp(\ttime_{n+1}))
\\[1.1em]
&\,+\,\gamma\bE^\perp(\htime_{n+1},\by_\eps(\htime_{n+1})\,+\,\eps\,\bv_\eps^\perp(t_{n+1}))\,,
\\[1.2em]
\bftau_{\bfv}^n\,:=\, &
\ds\frac{\bv_\eps(t_{n+1})-\bv_\eps(t_{n})}{\Dt}\,-\,(1-\gamma)\dfrac{\bE(t_n,\by_\eps(t_n)+\eps\,
  \bv_\eps^\perp(t_n))}{\eps}\,+\,(1-\gamma)\dfrac{\bv_\eps^\perp(\ttime_{n+1})}{\eps^2}
\\[1.1em]
&\,-\,\gamma\dfrac{\bE(\htime_{n+1},\by_\eps(\htime_{n+1})+\eps\, \bv_\eps^\perp(\htime_{n+1}))}{\eps}\,+\,\gamma\dfrac{\bv_\eps^\perp(t_{n+1})}{\eps^2}\,.
\end{cases}}

{The following lemma provides bounds for these local truncation errors.}

\begin{lemma}\label{l:2nd_consistency}
Assume $\bE\in W^{2,\infty}$. Then, for any $n\geq0$, $\eps>0$ and $\Dt>0$, {the consistency errors of stages \eqref{eq:lstab_stg1}--\eqref{eq:lstab_stg2-1} satisfy}
\mmn{
\begin{cases}
\|\ttauv^n\|
&\ds\underset{{\scriptscriptstyle {\bE}}}{\lesssim}
\frac{\Dt}{\eps^4}\ e^{K_{x}\,\ttime_{n+1}}\ (1+\eps^2)\,
\Big( \|\bv^0_\eps\|+\eps\,(1+\ttime_{n+1})\Big)\,,
\\[1.1em]
\|\htauy^n\|
&\ds\underset{{\scriptscriptstyle {\bE}}}{\lesssim}
\frac{\Dt}{\eps}\ e^{K_{x}\,\htime_{n+1}}\ 
\Big( \|\bv^0_\eps\|+\eps\,(1+\htime_{n+1})\Big)\,,
\\[1.1em]
\|\htauv^n\|
&\ds\underset{{\scriptscriptstyle {\bE}}}{\lesssim}
\frac{\Dt}{\eps^4}\ e^{K_{x}\,\htime_{n+1}}\ (1+\eps^2)\,
\Big( \|\bv^0_\eps\|+\eps\,(1+\htime_{n+1})\Big)\,,
\end{cases}}
and {the consistency errors of the final stage \eqref{eq:lstab_stg2-0} satisfy} 
\mmn{
\begin{cases}\ds
\|\bftau_{\by}^n\|
&\ds\underset{{\scriptscriptstyle {\bE}}}{\lesssim} 
\frac{(\Dt)^2}{\eps^3}e^{K_{x}\,\htime_{n+1}}\Big( \|\bv^0_\eps\|+\eps\,(1+\htime_{n+1})\Big)\,
\left(1
+\eps\,e^{K_{x}\,\htime_{n+1}}\,\Big( \|\bv^0_\eps\|+\eps\,(1+\htime_{n+1})\Big)\right)\,,
\\[1.1em]
\ds\|\bftau_{\bv}^n\|
&\ds\underset{{\scriptscriptstyle {\bE}}}{\lesssim} 
\frac{(\Dt)^2}{\eps^6}e^{K_{x}\,\htime_{n+1}}\Big( \|\bv^0_\eps\|+\eps\,(1+\htime_{n+1})\Big)
\left(1
+\eps^3\,e^{K_{x}\,\htime_{n+1}}\,\left( \|\bv^0_\eps\|+\eps\,(1+\htime_{n+1})\right)\right)\,.
\end{cases}}
\end{lemma}

\begin{proof}
We skip the proofs of the estimates on $(\ttauv^n,\htauy^n,\htauv^n)$ as almost identical to those in Lemma~\ref{l:1st_consistency}. {Concerning $(\bftau_{\by}^n,\bftau_{\by}^n)$, though the approach we adopt is also conceptually similar, it is obviously more technicality-laden. Firstly, from Taylor expansions stem}
\begin{align*}\ds
\ds\Big\|\frac{\by_\eps(t_{n+1})-\by_\eps(t_{n})}{\Dt}
-\by_\eps'(t_n)-\frac{\Dt}{2}\by_\eps''(t_n)\Big\|&\ds\leq \frac{(\Dt)^2}{6} \max_{[t_n,t_{n+1}]}\|\by_\eps'''\|\\
&\ds\,\underset{{\scriptscriptstyle {\bE}}}{\lesssim}
\frac{(\Dt)^2}{\eps^3}(1+\eps^2)\left(\eps+\max_{[t_n,t_{n+1}]}\|\bv_\eps\|\right)
+\frac{(\Dt)^2}{\eps^2}\max_{[t_n,t_{n+1}]}\|\bv_\eps\|^2,\\[2ex]
\ds\Big\|\frac{\bv_\eps(t_{n+1})-\bv_\eps(t_{n})}{\Dt}
-\bv_\eps'(t_n)-\frac{\Dt}{2}\bv_\eps''(t_n)\Big\|\ds&\leq \frac{(\Dt)^2}{6}	 \max_{[t_n,t_{n+1}]}\|\bv_\eps'''\|\,,\\
&\ds\underset{{\scriptscriptstyle {\bE}}}{\lesssim}
\frac{(\Dt)^2}{\eps^6}(1+\eps^4)\left(\eps+\max_{[t_n,t_{n+1}]}\|\bv_\eps\|\right)
+\frac{(\Dt)^2}{\eps^3}\max_{[t_n,t_{n+1}]}\|\bv_\eps\|	^2\,,
\end{align*}
{where the bounds on $\by_\eps'''$ and $\bv_\eps'''$ have been obtained from \eqref{e:Newton}.}

Likewise
\begin{align*}\ds
\left\|
(1-\gamma)\bv_\eps(\ttime_{n+1})
+\gamma\bv_\eps(t_{n+1}) -\bv_\eps(t_n)
-\frac{\Dt}{2}\bv_\eps'(t_n)
\right\| & \ds\underset{{\scriptscriptstyle {\bE}}}{\lesssim} \,(\Dt)^2\max_{[t_n,t_{n+1}]}\|\bv_\eps''\|,\\
&\ds\underset{{\scriptscriptstyle {\bE}}}{\lesssim}\,
\frac{(\Dt)^2}{\eps^4}(1+\eps^2)\left(\eps+\max_{[t_{n},t_{n+1}]}\|\bv_\eps\|\right)\,.
\end{align*}
With a bit more manipulations, we also derive on one hand
\begin{align*}\ds
\Big\|
\bE(t_n,\by_\eps(t_n)\,+\,\eps\,\bv_\eps^\perp(\ttime_{n+1}))
&\ds-\bE(t_n,\bx_\eps(t_n))
-\gamma\eps\Dt\,\dD_\bx\bE(t_n,\bx_\eps(t_n))\,(\bv_\eps^\perp)'(t_n)
\Big\| 
\\
&\ds\underset{{\scriptscriptstyle {\bE}}}{\lesssim} \eps^2(\Dt)^2\left(\max_{[t_n,\ttime_{n+1}]}\|\bv_\eps'\|\right)^2 
+\eps(\Dt)^2\max_{[t_n,\ttime_{n+1}]}\|\bv_\eps''\|\\
&\ds\underset{{\scriptscriptstyle {\bE}}}{\lesssim}
\frac{(\Dt)^2}{\eps^2}\left(\max_{[t_n,\ttime_{n+1}]}\|\bv_\eps\|\right)^2
+\frac{(\Dt)^2}{\eps^3}(1+\eps^2)\left(\eps+\max_{[t_n,\ttime_{n+1}]}\|\bv_\eps\|\right),
\end{align*}
on the other hand
\begin{align*}\ds
\Big\|
\bE(\htime_{n+1},\by_\eps(\htime_{n+1})\,+\,\eps\,\bv_\eps^\perp(t_{n+1}))
&\ds-\bE(\htime_{n+1},\bx_\eps(\htime_{n+1}))
-\left(1-\frac{1}{2\gamma}\right)\eps\Dt\,\dD_\bx\bE(t_{n},\bx_\eps(t_{n}))\,(\bv_\eps^\perp)'(t_n)
\Big\|\\
&\ds\underset{{\scriptscriptstyle {\bE}}}{\lesssim}
\frac{(\Dt)^2}{\eps^2}\left(\max_{[t_{n},\htime_{n+1}]}\|\bv_\eps\|\right)^2
+\frac{(\Dt)^2}{\eps^3}(1+\eps^2)\left(\eps+\max_{[t_{n},\htime_{n+1}]}\|\bv_\eps\|\right)\,,
\end{align*}
since
$$
\Big\|\dD_\bx\bE(\htime_{n+1},\bx_\eps(\htime_{n+1}))-\dD_\bx\bE(t_n,\bx_\eps(t_{n}))\Big\|
  \underset{{\scriptscriptstyle {\bE}}}{\lesssim} \frac{\Dt}{\eps} \left(\eps+\max_{[t_{n},\htime_{n+1}]}\|\bv_\eps\|\right)\,,
$$
and lastly
\begin{align*}\ds
\Big\|
\bE(\htime_{n+1},\bx_\eps(\htime_{n+1}))
&\ds-\bE(t_n,\bx_\eps(t_{n}))
-\frac{\Dt}{2\gamma}\d_t\bE(t_n,\bx_\eps(t_{n}))
-\frac{\Dt}{2\gamma}\dD_\bx\bE(t_n,\bx_\eps(t_{n}))\bx'_\eps(t_{n})
\Big\|\\
&\ds\underset{{\scriptscriptstyle {\bE}}}{\lesssim}
(\Dt)^2\left[1+\left(\max_{[t_n,\htime_{n+1}]}\|\bx_\eps'\|\right)^2
+\max_{[t_n,\htime_{n+1}]}\|\bx_\eps''\|\right]\\
&\ds\underset{{\scriptscriptstyle {\bE}}}{\lesssim}
\frac{(\Dt)^2}{\eps^2}\left(\max_{[t_n,\htime_{n+1}]}\|\bv_\eps\|\right)^2
+\frac{(\Dt)^2}{\eps^3}(1+\eps^2)\left(\eps+\max_{[t_{n},\htime_{n+1}]}\|\bv_\eps\|\right)\,.
\end{align*}

{Finally, by using the definition of $\gamma$ to derive the following identities}
\[
(1-\gamma)\,\gamma\,+\,\gamma\,\left(1-\frac{1}{2\gamma}\right)\,=\,0\,,\qquad\qquad
(1-\gamma)\,\gamma\,+\,\gamma\,=\,\frac{1}{2}\,,
\]
{one may combine all the estimates with \eqref{e:Newton} to obtain}
\mmn{
\begin{cases}\ds
\|\bftau_{\by}^n\|
&\ds\,\underset{{\scriptscriptstyle {\bE}}}{\lesssim}\, 
\frac{(\Dt)^2}{\eps^3}(1+\eps^2)\left(\eps+\max_{[t_n,\htime_{n+1}]}\|\bv_\eps\|\right)
+\frac{(\Dt)^2}{\eps^2}\,\left(\max_{[t_n,\htime_{n+1}]}\|\bv_\eps\|\right)^2,
\\[1.1em]
\ds
\|\bftau_{\bv}^n\|
&\ds\,\underset{{\scriptscriptstyle {\bE}}}{\lesssim}\, 
\frac{(\Dt)^2}{\eps^6}(1+\eps^4)\left(\eps+\max_{[t_n,\htime_{n+1}]}\|\bv_\eps\|\right)
+\frac{(\Dt)^2}{\eps^3}\,\left(\max_{[t_n,\htime_{n+1}]}\|\bv_\eps\|\right)^2,
\end{cases}}
and the proof is concluded by applying Lemma~\ref{lem:apriori}.
\end{proof}

{As we discussed in \S\ref{sec:3.2}, in addition to the consistency estimates of the foregoing lemma, one needs a stability analysis. To investigate the stability of the implicit part, we observe that combining velocity updates of the scheme \eqref{eq:lstab_stg}, employing the explicit expression of $\bF^\eps_{\bv}$, and manipulating the terms, one derives
\mmn{
\bv_\eps^{n+1}=(\Id+\gamma\,\lambda\,\J)^{-1}\left(\Id-(1-\gamma)\J\,(\Id+\gamma\,\lambda\,\J)^{-1}\right)\bv_\eps^n+ \cdots,
}
with $\lambda=\Dt/\eps^2$. Hence the need to investigate the stability of the matrix $\A_\lambda$ defined as 
\[
\A_\lambda:=(\Id+\gamma\,\lambda\,\J)^{-1}(\Id-(1-\gamma)\lambda\J\,(\Id+\gamma\,\lambda\,\J)^{-1})
=(\Id+\gamma\,\lambda\,\J)^{-2}(\Id+(2\gamma-1)\lambda\J).
\]
}

\begin{lemma}\label{l:2nd-J} Let $\J$ be as in \eqref{rotation:mat}. Then, for any $\lambda>0$, the matrix $\A_\lambda$ satisfies
\mmn{
\left\|\A_\lambda\right\|
=\frac{1}{\sqrt{1+\frac{\gamma^4\lambda^4}{1+2\gamma^2\lambda^2}}}<1\,.
}
\end{lemma}
\begin{proof}
Proceeding as in the proof of Lemma~\ref{l:1st-J} yields
\[
\left\|(\Id+\gamma\,\lambda\,\J)^{-2}(\Id+(2\gamma-1)\lambda\J)\right\|
=\frac{\sqrt{1+(2\gamma-1)^2\lambda^2}}{1+\gamma^2\lambda^2}\,.
\]
Since from the equation defining $\gamma$ stems $(2\gamma-1)^2=2\gamma^2$, this achieves the proof.
\end{proof}

\begin{proposition}
\label{prop:E1_y_2nd}
{There exists a constant $C_0>0$ such that, when $\bE\in W^{2,\infty}$,  the error from the unique solution of the scheme \eqref{eq:lstab_stg} to the exact solution of the system \eqref{e:Newton} is such that for any $\Dt>0$, $n\geq0$ and $\eps>0$}
\begin{align*}
\|\bx_\eps^n&-\bx_\eps(t_n)\|
\,\leq\,\|\by_\eps^n-\by_\eps(t_n)\|
+\eps\|\bv_\eps^n-\bv_\eps(t_n)\|\,\\[0.5em]
&\underset{{\scriptscriptstyle {\bE}}}{\lesssim}\,
\frac{(\Dt)^2}{\eps^5}\,t_n\,e^{C_0K_xt_n}
\Big( \|\bv^0_\eps\|+\eps\,(1+t_{n})\Big)
\Big(1+\eps^3\,\left(\|\bv^0_\eps\|+\eps\,(1+t_{n})\right)\Big)\,.
\nonumber
\end{align*}
\end{proposition}

\begin{proof}
As in the proof of Proposition~\ref{prop:E1_y}, we introduce {numerical errors as}
$$
\ber_{\by}^n\,:=\,\by_\eps(t_n)-\by_\eps^n\,, \qquad \qquad\ber_{\bv}^n\,:=\,\bv_\eps(t_n)-\bv_\eps^n\,,
$$
and likewise for the intermediate stages 
\mmn{
\ber_{\tbv}^n:=\,\bv_\eps(\ttime_n)-\tbv_\eps^n\,,\qquad\qquad
\ber_{\hy}^n:=\,\by_\eps(\htime_n)-\hy_\eps^n\,,\qquad\qquad
\ber_{\hv}^n:=\,\bv_\eps(\htime_n)-\hv_\eps^n\,.
}
We also use the norm $\|\cdot\|_\eps$ from \eqref{d:esp-norm} for our
estimates and set $\lambda=\Dt/\eps^2$.

On the one hand, by applying Lemma~\ref{l:1st-J},  we obtain from the first
stage \eqref{eq:lstab_stg1}
$$
\eps\left\|\ber_{\tbv}^{n+1}
-(\Id+\gamma\lambda\J)^{-1}\ber_{\bv}^n\right\|
\,\lesssim\,
\min\left(1,\frac1\lambda\right)
\,\left(K_x\Dt\,\|(\ber_{\by}^{n},\ber_{\bv}^{n})\|_\eps
+\Dt\,\eps\,\|\ttauv^n\|\right)\,,
$$
and from the intermediate stage \eqref{eq:lstab_stg2-1}
\begin{align*}
\ds\|(\ber_{\hy}^{n+1},\ber_{\hv}^{n+1})\|_\eps
&\,\ds\lesssim
  (1+K_x\Dt)\,\|(\ber_{\by}^{n},\ber_{\bv}^{n})\|_\eps+\Dt\,\|(\htauy^{n},\htauv^{n})\|_\eps
\\[0.5em]
&\,\ds\quad+\,(\lambda+K_x\Dt)\,\eps\|\ber_{\tbv}^{n+1}-(\Id+\gamma\lambda\J)^{-1}\ber_{\bv}^n\|\,.
\end{align*}
\begin{align*}
\ds\|(\ber_{\hy}^{n+1},\ber_{\hv}^{n+1})\|_\eps
&\,\ds\lesssim
  (1+K_x\Dt)^2\,\|(\ber_{\by}^{n},\ber_{\bv}^{n})\|_\eps+\Dt\,\|(\htauy^{n},\htauv^{n})\|_\eps
+\Dt\,(1+K_x\Dt)\,\eps\,\|\ttauv^n\|\,.
\end{align*}
On the other hand, by using Lemmas~\ref{l:1st-J} again, the last stage \eqref{eq:lstab_stg2-0} yields  both
\begin{align*}
\|\ber_{\by}^{n+1}\|-\|\ber_{\by}^{n}\|
\ds\,\lesssim\,&
K_x\Dt\,\|(\ber_{\by}^{n},\ber_{\bv}^{n})\|_\eps
+K_x\Dt\|(\ber_{\hy}^{n+1},\ber_{\hv}^{n+1})\|_\eps+\Dt\,\|\bftau_{\by}^{n}\|
\\[0.5em]
&\;+\;K_x\Dt\,\eps\|\ber_{\tbv}^{n+1}-(\Id+\gamma\lambda\J)^{-1}\ber_{\bv}^n\|
\end{align*}
and
\begin{align*}
\eps\|\ber_{\bv}^{n+1}-\A_{\lambda}\,\ber_{\bv}^{n}\|
&\lesssim
K_x\Dt\,\|(\ber_{\by}^{n},\ber_{\bv}^{n})\|_\eps
+K_x\Dt\|(\ber_{\hy}^{n+1},\ber_{\hv}^{n+1})\|_\eps+\Dt\,\eps\|\bftau_{\bv}^{n}\|
\\[0.5em]
&\,+\,\Dt\left(\frac{1}{\eps^2}+K_x\right)\,\eps\|\ber_{\tbv}^{n+1}-(\Id+\gamma\lambda\J)^{-1}\ber_{\bv}^n\|\,.
\end{align*}

By gathering the foregoing estimates and using Lemma~\ref{l:2nd-J} we deduce for $n\geq0$
\begin{align*}\ds
\|(\ber_{\by}^{n+1},\ber_{\bv}^{n+1})\|_\eps
-\|(\ber_{\by}^{n},\ber_{\bv}^{n})\|_\eps
\ds
\lesssim \, &K_x\Dt\,(1+K_x\Dt)^2\,\|(\ber_{\by}^{n},\ber_{\bv}^{n})\|_\eps
+\Dt\,\|(\bftau_{\by}^{n},\bftau_{\bv}^{n})\|_\eps\\
&\ds \,+\,K_x (\Dt)^2\,\|(\htauy^{n},\htauv^{n})\|_\eps
+(\Dt)^2\frac{1+K_x\eps^2}{\eps^2}\,(1+K_x\Dt)\,\eps\,\|\ttauv^n\|.
\end{align*}
The proof is then concluded as in the proof of Proposition~\ref{prop:E1_y_y}, by application of the discrete Gr\"onwall lemma and Lemma~\ref{l:2nd_consistency}.
\end{proof}

{We, then, continue with the direct convergence analysis of \eqref{e:dNgyro-2nd_1}--\eqref{e:dNgyro-2nd} to \eqref{e:Ngyro}, which is the counterpart of Proposition \ref{prop:E1_y_2nd} for the asymptotic model.}

\begin{proposition}
\label{prop:E1_gc_2nd}
There exists a constant $C_0>0$ such that when $\bE\in W^{2,\infty}$, we have 
\mmn{
\|\bx^n-\bx(t_n)\|
&\underset{{\scriptscriptstyle {\bE}}}{\lesssim}\, (\Dt)^2\,t_n\,e^{C_0K_x\,t_n}\,,
}
when $(\bx^{n})_{n\in\N}$ and $\bx$ solve respectively \eqref{e:dNgyro-2nd_1}--\eqref{e:dNgyro-2nd} and \eqref{e:Ngyro}, with the same initial datum.
\end{proposition}
\begin{proof}
The proof is omitted as essentially a simpler version of the proof of Proposition~\ref{prop:E1_y_2nd}.
\end{proof}
 
\subsection{Asymptotic  estimates}
\label{sec:5.3}

{Regarding the asymptotic part of the convergence analysis, and to prepare the comparisons between solutions of \eqref{eq:lstab_stg1}--\eqref{eq:lstab_stg2-0} and \eqref{e:dNgyro-2nd_1}--\eqref{e:dNgyro-2nd}, we now examine $\eps$-uniform boundedness} of solution of the scheme \eqref{eq:lstab_stg1}--\eqref{eq:lstab_stg2-0}. To do so, as we have done in previous sections, we work with auxiliary variables that are small corrections to stiff velocity variables. They are defined as
\[
\tbz^n_\eps:=\tbv^n_\eps
\,+\,\eps\,\bE^\perp(t_{n-1},\by_\eps^{n-1}+\eps(\bv_\eps^{n-1})^\perp)
\,=\,\eps\bJ\bF^\eps_{\bv}(t_{n-1},\by_\eps^{n-1},\eps\,\bv_\eps^{n-1},\eps\,\tbv_\eps^{n}),
\qquad n\geq 1,
\]
{for the correction to the (intermediate) updated velocity $\tbv^n_\eps$, and} 
\[
\bz^n_\eps
:=\bv^n_\eps
\,+\,\eps\,\bE^\perp(\htime_{n},\hy_\eps^{n}+\eps(\hv_\eps^{n})^\perp)
\,=\,\eps\bJ\bF^\eps_{\bv}(\htime_{n},\hy_\eps^{n},\eps\,\hv_\eps^{n},\eps\,\bv_\eps^{n})
\,,\qquad\qquad n\geq 1\,,
\]
{for the correction to the (final) updated velocity $\bv^n_\eps$. By employing the scheme \eqref{eq:lstab_stg1}--\eqref{eq:lstab_stg2-0}, with $\lambda=\Dt/\eps^2$, these definitions imply the following updates for $n\geq1$}
\begin{align*}
\tbz^{n+1}_\eps&-(\Id+\gamma\lambda\J)^{-1}\bz^{n}_\eps
=\eps\,(\Id+\gamma\lambda\J)^{-1}\,
\left(\bE^\perp(t_{n},\by_\eps^{n}+\eps(\bv_\eps^{n})^\perp)
-\bE^\perp(\htime_{n},\hy_\eps^{n}+\eps(\hv_\eps^{n})^\perp)\right)\\
\bz^{n+1}_\eps&-\bA_\lambda\bz^{n}_\eps
=-(1-\gamma)\lambda\bJ\,(\Id+\gamma\lambda\J)^{-1}\,
\left(\tbz^{n+1}_\eps-(\Id+\gamma\lambda\J)^{-1}\bz^{n}_\eps\right)\\
&\qquad\qquad\quad
+\eps\,(\Id+\gamma\lambda\J)^{-1}\,
\left(\bE^\perp(\htime_{n+1},\hy_\eps^{n+1}+\eps(\hv_\eps^{n+1})^\perp)
-\bE^\perp(\htime_{n},\hy_\eps^{n}+\eps(\hv_\eps^{n})^\perp)\right)
\end{align*}
and the initial values
\begin{align*}
\tbz^{1}_\eps&-(\Id+\gamma\lambda\J)^{-1}\bv^{0}_\eps
=\eps\,(\Id+\gamma\lambda\J)^{-1}\,
\bE^\perp(t_{0},\by_\eps^{0}+\eps(\bv_\eps^{0})^\perp)\\
\bz^{1}_\eps&-\bA_\lambda\bv^{0}_\eps
=-(1-\gamma)\lambda\bJ\,(\Id+\gamma\lambda\J)^{-1}\,
\left(\tbz^{1}_\eps-(\Id+\gamma\lambda\J)^{-1}\bv^{0}_\eps\right)
+\eps\,(\Id+\gamma\lambda\J)^{-1}\,
\bE^\perp(\htime_{1},\hy_\eps^{1}+\eps(\hv_\eps^{1})^\perp)\,.
\end{align*}

\begin{lemma}
\label{lem:disct_zbound_y2}
There exists a constant $C_0>0$ such that, when $\bE\in W^{1,\infty}$, for $n\geq 1$,
\mmn{
\|\tbv_\eps^n\|
&\ds\,\lesssim\,\|\tbz_\eps^n\|\,+\,K_0\,\eps\,,&
\|\bv_\eps^n\|
&\ds\,\lesssim\,\|\bz_\eps^n\|\,+\,K_0\,\eps\,,\\[0.55em]
\ds\|\tbz_\eps^n\|
&\ds\,\lesssim\,\min\left(1,\frac{\eps^2}{\Dt}\right)\left(\|\bv_\eps^{n-1}\|\,+\,K_0\,\eps\right)\,,&
\ds\|\hv_\eps^n\|
&\ds\,\lesssim\,\|\bv^{n-1}_\eps\|+K_0\,\eps\,,
}
and
\mmn{
\ds\|\bz_\eps^n\|
\ds\lesssim\,
e^{C_0K_{x}\,t_n}
\left(\|\bv^0_\eps\|+\eps\,K_0+\eps\,t^n(K_t+K_xK_0)\right)\,.
}
\end{lemma}
\begin{proof}
It follows from \eqref{eq:lstab_stg2-1}--\eqref{eq:lstab_stg2-0} that for $n\geq1$,
\[
\|\by^n_\eps-\hy^{n}_\eps\|\lesssim\,K_0\Dt\,,\qquad\qquad
\|\hy^{n+1}_\eps-\hy^{n}_\eps\|\lesssim\,K_0\Dt\,,\qquad\qquad
\|\bv^n_\eps-\hv^{n}_\eps\|\lesssim\,\frac{\Dt}{\eps^2}
\left(\|\bz^n_\eps\|+\|\tbz^n_\eps\|\right)\,,
\]
and
\[
\|\hv^{n+1}_\eps-\hv^{n}_\eps\|\lesssim\,\frac{\Dt}{\eps^2}
\left(\|\bz^n_\eps\|+\|\tbz^n_\eps\|+\|\tbz^{n+1}_\eps\|\right)
\lesssim\,\frac{\Dt}{\eps^2}
\left(\|\bz^n_\eps\|+\|\tbz^n_\eps\|+K_0\eps\right)\,.
\]
Thus for some $c_0>0$, for $n\geq1$,
\[\begin{array}{rl}
\ds\|\tbz^{n+1}_\eps\|
\ds-\|\bz^n_\eps\|
&\lesssim 
K_x\Dt\,(\|\bz^n_\eps\|+\|\tbz^n_\eps\|)+\eps\Dt (K_t+K_xK_0)\,,\\[0.5em]
\ds\|\bz^{n+1}_\eps\|
\ds-\|\bz^n_\eps\|
-c_0\,K_x\Dt\|\tbz^n_\eps\|
&\lesssim 
K_x\Dt\,\|\bz^n_\eps\|+\eps\Dt (K_t+K_xK_0)\,,
\end{array}\]
that may be combined to give for $n\geq1$,
\begin{align*}
(\|\bz^{n+1}_\eps\|+c_0K_x\Dt\|\tbz^{n+1}_\eps\|)
&-(\|\bz^n_\eps\|+c_0K_x\Dt\|\tbz^n_\eps\|)\\
&\lesssim\,K_x\Dt(1+K_x\Dt)(\|\bz^n_\eps\|+c_0K_x\Dt\|\tbz^n_\eps\|)
+\eps\Dt (1+K_x\Dt)(K_t+K_xK_0)\,.
\end{align*}
At this stage completing an application of the discrete Gr\"onwall lemma with the initial bounds 
\[
\|\tbz^1_\eps\|\lesssim \|\bv^0_\eps\|+K_0\eps\,,\qquad
\|\bz^1_\eps\|\lesssim \|\bv^0_\eps\|+K_0\eps\,.
\]
achieves the proofs of the bound on $\bz^n_\eps$.

Bounds on $\bv^n_\eps$ and $\tbv^n_\eps$ are obvious from the definitions, bounds on $\tbz^n_\eps$ and $\hv^n_\eps$ follow from \eqref{eq:lstab_stg1}-\eqref{eq:lstab_stg2-1} and Lemma~\ref{l:1st-J}.
\end{proof}

{Now, to state a comparison result we introduce notation $\bX^{\Dt}$ and $(\bX_\eps^{\Dt},\bV_\eps^{\Dt})$ to denote discrete flows for \eqref{e:dNgyro-2nd_1}-\eqref{e:dNgyro-2nd} and \eqref{eq:lstab_stg1}--\eqref{eq:lstab_stg2-0}.} We also set $\bY_\eps^{\Dt}:=\bX_\eps^{\Dt}-\eps\,(\bV_\eps^{\Dt})^\perp$.

\begin{proposition} 
\label{prop:disct_err_2nd}
\begin{enumerate}[(i)]
\item
There exists a constant $C_0>0$ such that when $\bE\in W^{1,\infty}$
\begin{align*}\ds
\|\bX_\eps^{\Dt}(t_n,0,\bx_\eps^0,\bv_\eps^0)\ds-\bX^{\Dt}(t_n,0,\bx_\eps^0)\|\ds\underset{{\scriptscriptstyle {\bE}}}{\lesssim}
\eps\,e^{C_0K_{x}\,t_n}
(1+t^n)\left(\|\bv^0_\eps\|+\eps\right)
\,.
\end{align*}
\item 
There exists a constant $C_0>0$ such that when $\bE\in W^{2,\infty}$
\begin{align*}
\|\bY_\eps^{\Dt}(t_n,0,\bx_\eps^0,\bv_\eps^0)
-\,\bX^{\Dt}(t_n,0,\bx_\eps^0-\eps(\bv_\eps^0)^\perp)\|
\ds\underset{{\scriptscriptstyle {\bE}}}{\lesssim}
\eps^2\,e^{C_0K_{x}\,t_n}\,\left(1+t_n\right)^3
\,\big(1+\|\bv_\eps^0\|^2+\eps^2\big)\,.
\end{align*}
\end{enumerate}
\end{proposition}
\begin{proof}
Along the proof we use the notational conventions introduced in the proof of Proposition~\ref{prop:disct_err} and variations thereof. 

As before, concerning the first estimate, we sum differences between respective equations and observe that for $n\geq0$
\[
\|\bx_\eps^n-\bx^n\|
\lesssim \eps\,[\|\bv^{n}_\eps\|+\|\bv^{0}_\eps\|]
+K_x\Dt\sum_{\ell=0}^{n-1}\|\bx_\eps^\ell-\bx^\ell\|
+K_x\Dt\sum_{\ell=1}^{n}\|\hx_\eps^\ell-\hx^\ell\|
\]
where for $n\geq1$, $\hx_\eps^n:=\hy_\eps^n+ \eps(\hv_\eps^n)^\perp$. Now,  from
\eqref{eq:lstab_stg1} and \eqref{e:dNgyro-2nd_1} follows for $n\geq1$
\[
\|\hx_\eps^n-\hx^n\|\lesssim
\|\bx_\eps^{n-1}-\bx^{n-1}\|\left(1+K_x\Dt\right)
+\eps\,[\|\hv^{n}_\eps\|+\|\bv^{n-1}_\eps\|]
+K_x\eps\Dt[\|\tbv^{n}_\eps\|+\|\bv^{n-1}_\eps\|]\,.
\]
Hence for $n\geq 0$
\begin{align*}\ds
\|\bx_\eps^{n}-\bx^{n}\|
\ds\,&\lesssim 
K_x\Dt\left(1+K_x\Dt\right)\sum_{\ell=0}^{n-1}\|\bx_\eps^{\ell}-\bx^{\ell}\|\\
&+\eps\,(\|\bv_\eps^n\|+\|\bv_\eps^0\|)
+\eps\,K_x\Dt\left(1+K_x\Dt\right)\sum_{\ell=1}^{n}(\|\hv^{\ell}_\eps\|+\|\tbv^{\ell}_\eps\|+\|\bv^{\ell-1}_\eps\|)\,.
\end{align*}
{Finally, applying the discrete Gr\"onwall lemma \eqref{eq:dgronwall1}--\eqref{eq:dgronwall2}, combined with the velocity bound of Lemma~\ref{lem:disct_zbound_y2} concludes the proof of the first inequality.}

{For the second estimate, one begins with the Taylor expansion of the $\by_\eps$-update \eqref{eq:lstab_stg2-0}, that is for $n\geq 0$,}
\begin{align*}\ds
\frac{\by_\eps^{n+1}-\by_\eps^n}{\Dt}
&\ds
+\,(1-\gamma)\bE^\perp(t_n,\by_\eps^n)+\,\gamma\,\bE^\perp(\htime_{n+1},\hy_\eps^{n+1})\\
&\ds
=\,-\,\eps\,\dD_x\bE^\perp(\htime_{n+1},\hy_\eps^{n+1})\left(
\gamma(\bv_\eps^{n+1})^\perp+(1-\gamma)(\tbv_\eps^{n+1})^\perp\right)\\
&\quad\,-\,\eps\,(1-\gamma)\,
\left(\dD_x\bE^\perp(t_n,\by_\eps^n)-\dD_x\bE^\perp(\htime_{n+1},\hy_\eps^{n+1})\right)(\tbv_\eps^{n+1})^\perp
\\
&\ds
\quad+\eps^2\left((1-\gamma)\Theta_\eps(t_n,\by_\eps^n,\tbv_\eps^{n+1})
+\gamma\,\Theta_\eps(\htime_{n+1},\hy_\eps^{n+1},\bv_\eps^{n+1})
\right)\,,
\end{align*}
with $\Theta_\eps$ such that $\|\Theta_\eps(t,\by,\bv)\|\leq \tfrac12 K_{xx}\|\bv\|^2$. {To see that the third line also possesses an $\eps^2$-bound, on may combine Lemma~\ref{l:1st-J} that gives for any $n\geq 0$,}
$$
\|\tbv^{n+1}_\eps\|\,\lesssim\,K_0\eps+\frac{\eps^2}{\Dt}\|\bv^{n}_\eps\|
$$
with \eqref{eq:lstab_stg2-1} that yields
$$
\Big\|\dD_x\bE(t_n,\by_\eps^n)-\dD_x\bE(\htime_{n+1},\hy_\eps^{n+1})\Big\|
\lesssim \min\left(K_x,\Dt\,\left(K_{tx}+K_{xx}K_0\right)\right)\,.
$$
Therefore, one may focus on the second line. We observe that, when $n\geq 1$,
\begin{align*}\ds
\dD_x\bE(\htime_{n},\hy_\eps^{n})\left(
\gamma(\bv_\eps^{n})^\perp+(1-\gamma)(\tbv_\eps^{n})^\perp\right)\,=\,&\,-\frac{\eps^2}{\Dt}\left(\dD_x\bE(\htime_{n+1},\hy_\eps^{n+1})\bv_\eps^{n}-\dD_x\bE(\htime_{n},\hy_\eps^{n})\bv_\eps^{n-1}\right)\\
&\ds
+\eps\,\dD_{\bx}\bE(\htime_{n},\hy_\eps^{n})\,
\left(\gamma\bE(\htime_{n},\hy_\eps^{n})+(1-\gamma)\bE(t_{n-1},\bx_\eps^{n-1})\right)
\\
&\ds+\frac{\eps^2}{\Dt}\left(\dD_x\bE(\htime_{n+1},\hy_\eps^{n+1})-\dD_x\bE(\htime_{n},\hy_\eps^{n})\right)\,\bv_\eps^{n}\,,
\end{align*}
{with the following estimate concerning the last term}
\begin{align*}
\Big\|\dD_x\bE&\ds(\htime_{n+1},\hy_\eps^{n+1})-\dD_x\bE(\htime_{n},\hy_\eps^{n})\Big\|
\lesssim \Dt\,\left(K_{tx}+K_{xx}K_0\right)\,.
\end{align*}
As a consequence summing yields, for $n\geq 0$,
\begin{align*}
\|\by^{n}_\eps\ds-\by^{n}\|
&\lesssim
K_x\Dt\sum_{\ell=0}^{n-1}\|\by^{\ell}_\eps-\by^{\ell}\|
+K_x\Dt\sum_{\ell=1}^{n}\|\hy^{\ell}_\eps-\hy^{\ell}\|\\
&+\eps^3\,K_x\left(\|\bv^{n}_\eps\|+\|\bv^0_\eps\|\right)
+\eps^2\,t^n\,K_x\,K_0
+\eps^3\Dt\,\left(K_{tx}+K_{xx}K_0\right)\sum_{\ell=0}^{n}\|\bv^\ell_\eps\|\\
&+\eps^2\Dt\,K_{xx}\left(\sum_{\ell=0}^{n}\|\bv^\ell_\eps\|^2
+\sum_{\ell=1}^{n}\|\tbv^\ell_\eps\|^2\right)\,.
\end{align*}
Finally, we note that for $n\geq 1$
\[
\|\hy^{n}_\eps-\hy^{n}\|
\lesssim \|\by^{n-1}_\eps-\by^{n-1}\|\,(1+K_x\Dt)
+\eps^2\Dt\,K_x\,K_0+\eps^3\,K_x\,\|\bv^{n-1}_\eps\|\,.
\]

Inserting the latter in the former leaves, for $n\geq0$,
\begin{align*}
\|\by^{n}_\eps\ds-\by^{n}\|
&\lesssim
K_x\Dt(1+K_x\Dt)\sum_{\ell=0}^{n-1}\|\by^{\ell}_\eps-\by^{\ell}\|
+\eps^3\,K_x\left(\|\bv^{n}_\eps\|+\|\bv^0_\eps\|\right)\\
&
+\eps^2\,t^n\,K_x\,K_0\,(1+K_x\Dt)
+\eps^3\Dt\,\left(K_x^2+K_{tx}+K_{xx}K_0\right)\sum_{\ell=0}^{n}\|\bv^\ell_\eps\|\\
&+\eps^2\Dt\,K_{xx}\left(\sum_{\ell=0}^{n}\|\bv^\ell_\eps\|^2
+\sum_{\ell=1}^{n}\|\tbv^\ell_\eps\|^2\right)\,.
\end{align*}
The proof is again achieved by combining the discrete Gr\"onwall lemma with Lemma~\ref{lem:disct_zbound_y2}.
\end{proof}

\subsection{Proof of Theorem \ref{thm:error_estim_2nd}}
{The unique solvability of the scheme is again a direct consequence of Lemma~\ref{l:1st-J}. Regarding error estimates,  one gathers direct estimates in Proposition~\ref{prop:E1_y_2nd} with the combination of Propositions~\ref{prop:E1_gc_2nd} and~\ref{prop:disct_err_2nd} and Theorem \ref{thm:second_estim} to conclude the proof.}

\section{Numerical experiments}\label{sec:num_exp}
In this section, we provide an illustration of the error estimates proved in Theorems~\ref{thm:error_estim_1st_y} and~\ref{thm:error_estim_2nd} on the simple example of the motion of a single particle subject to an electric field $\bE=-\nabla_\bx \phi$ deriving from the potential
\mmn{
\phi(\bx)=\frac{1}{2}\left(\|\bx\|^2+\frac{1}{10\pi}\cos^2(2\pi x_2)\right),\qquad
\bx=(x_1,x_2)\,.
}
Initial conditions are chosen as
\mmn{
\bx^0_\eps=(1,1),\qquad \qquad
\bv^0_\eps=(3,3).
}

Note that for this electric potential, the electric potential does not fit exactly in the framework of our theorems since it is unbounded (though its derivatives from order one and onward are bounded). Yet our observations fit well with our theoretical conclusions.

We observe two error indicators for the variable $\by_\eps$,
\mmn{
\begin{cases}
\ds\mathcal{E}_{\by}(\Delta t,\eps) \,:=\, \sum_{n=1}^{N_T} \Dt\,\|\by_\eps^{n}-\by_\eps(t_n)\|,
\\[0.9em]
\ds\mathcal{E}_{\by,\,\rm{gc}}(\Delta t,\eps) \,:=\, \sum_{n=1}^{N_T} \Dt\,\|\by_\eps^{n}-\bX(t_n,0, \bx^0-\eps(\bv^0_\eps)^\perp) \|,
\end{cases}
}
where $\by_\eps:=\bx_\eps-\eps(\bv_\eps)^\perp$ stems from the solution $(\bx_\eps,\bv_\eps)$ of the system of characteristics \eqref{e:Newton} (with $s=0$) and $\bX$ is the flow for the guiding center equation \eqref{e:Ngyro}, whereas  $\by_\eps^n$ is our numerical approximation of $\by_\eps(t_n)$. {So, in other words, $\mathcal{E}_{\by}$ and $\mathcal{E}_{\by,\,\rm{gc}}$ measure the difference from the numerical approximation to, respectively, the exact $\eps$-dependent and asymptotic solutions, hence quantify respectively numerical convergence and asymptotic convergence.} Note that $\mathcal{E}_{\by}$ and $\mathcal{E}_{\by,\,\rm{gc}}$ are averaged errors in time so that they take into account the possibly-large errors which may originate from the initial layer. Similarly, we 
define errors for the velocity variable $\bv_\eps$ as
\mmn{
\begin{cases}
\ds\mathcal{E}_{\bv}(\Delta t,\eps) \,:=\, \sum_{n=1}^{N_T} \Dt\,\|\bv_\eps^{n}-\bv_\eps(t_n)\|,
\\[0.9em]
\ds\mathcal{E}_{\bv,\,\rm{gc}}(\Delta t,\eps) \,:=\, \sum_{n=1}^{N_T} \Dt\,\|\eps^{-1}\bv_\eps^{n}-\bv_{\rm gc}(t_n) \|,
\end{cases}
}
using the guiding center velocity $\bv_{\rm gc}(t):=-\bE^\perp(t,\bx_\eps(t))$ as the asymptotic velocity.

Since the exact solution of this example is not available, we perform a very accurate numerical simulation by a fourth-order explicit Runge--Kutta scheme as the reference solution. The time step for this reference solution is chosen small enough, namely of order $\eps^2$ when $\eps\ll 1$, so as to capture the very fast oscillations. 

We, then, perform some numerical experiments with the first-order scheme \eqref{e:dNewton-1st_y}  to illustrate the results stated in Theorems~\ref{thm:error_estim_1st_y}. 

In Fig.~\ref{fig:01}(A), we illustrate the computed errors for the first-order scheme \eqref{e:dNewton-1st_y}, which match the estimate of Theorem~\ref{thm:error_estim_1st_y}; one can identify in the figure, roughly speaking, the $min$ function of the estimate, \cf \cite[Fig.\ 1]{jin2010asymptotic}: for $\eps\sim 1$, the error grows as one decreases $\eps$ up to some turning point in the curve after which the error decreases with $\eps$ getting closer to zero. More precisely, one can observe that, when $\Dt$ is much smaller than $\eps$ (right part of Fig.\ \ref{fig:01}(A)), the error is $\cO(\Dt)$, \ie, the scheme is first-order accurate with respect to $\Dt$, as the classical analysis may suggest. Note that in this regime, the other error estimate in Theorem \ref{thm:error_estim_1st_y}, which behaves like $\eps^2+\Dt$, is quite large.  On the other hand, for smaller values of $\eps$, it is the latter bound which saturates the numerical error due to the blow-up of the classical error estimate, so the error is dominated by $\cO(\eps^2)$ as the slope of the error curve suggests. Indeed, we see this saturation in the intermediate region for $\eps\in[10^{-3},10^{-1}]$. Of course, when we refine the time step $\Dt$, the intermediate region moves to the left. Finally, when $\eps^2$ is very small compared to $\Dt$, the error in $\Dt$ dominates and the error does not decrease with $\eps$ any longer. 

Moreover, in Fig.\ \ref{fig:01}(B), we compare our numerical approximation with the resolved reference solution of the guiding center model $\bX(\cdot,0,\bx^0-\eps(\bv^0_\eps)^\perp$. This, in fact, confirms that when $\eps\ll 1$, the scheme is first-order accurate and corresponds to the left part of Fig.\ \ref{fig:01}(A), as one expects that the reference solution converges to its asymptotic limit.

We also present the numerical error on the velocity variable, in terms of $\mathcal{E}_{\bv}$ and $\mathcal{E}_{\bv,\,\rm{gc}}$. Fig.\ \ref{fig:01}(C) confirms the point that the scheme should be first-order with respect to $\Dt$, when $\eps\sim 1$. However, for $\eps\ll 1$ and a with large time step, the numerical scheme does not capture fast oscillations; so, no convergence to the reference velocity can be observed. In this regime, the computed velocity is only able to compute slow scale dynamics represented by the guiding center velocity $\bv_{\rm gc}$, as Fig.\ \ref{fig:01}(D) suggests.

 \begin{figure}[h!]
\begin{subfigure}[b]{0.45\textwidth}
\centering
 \includegraphics[width=\textwidth]{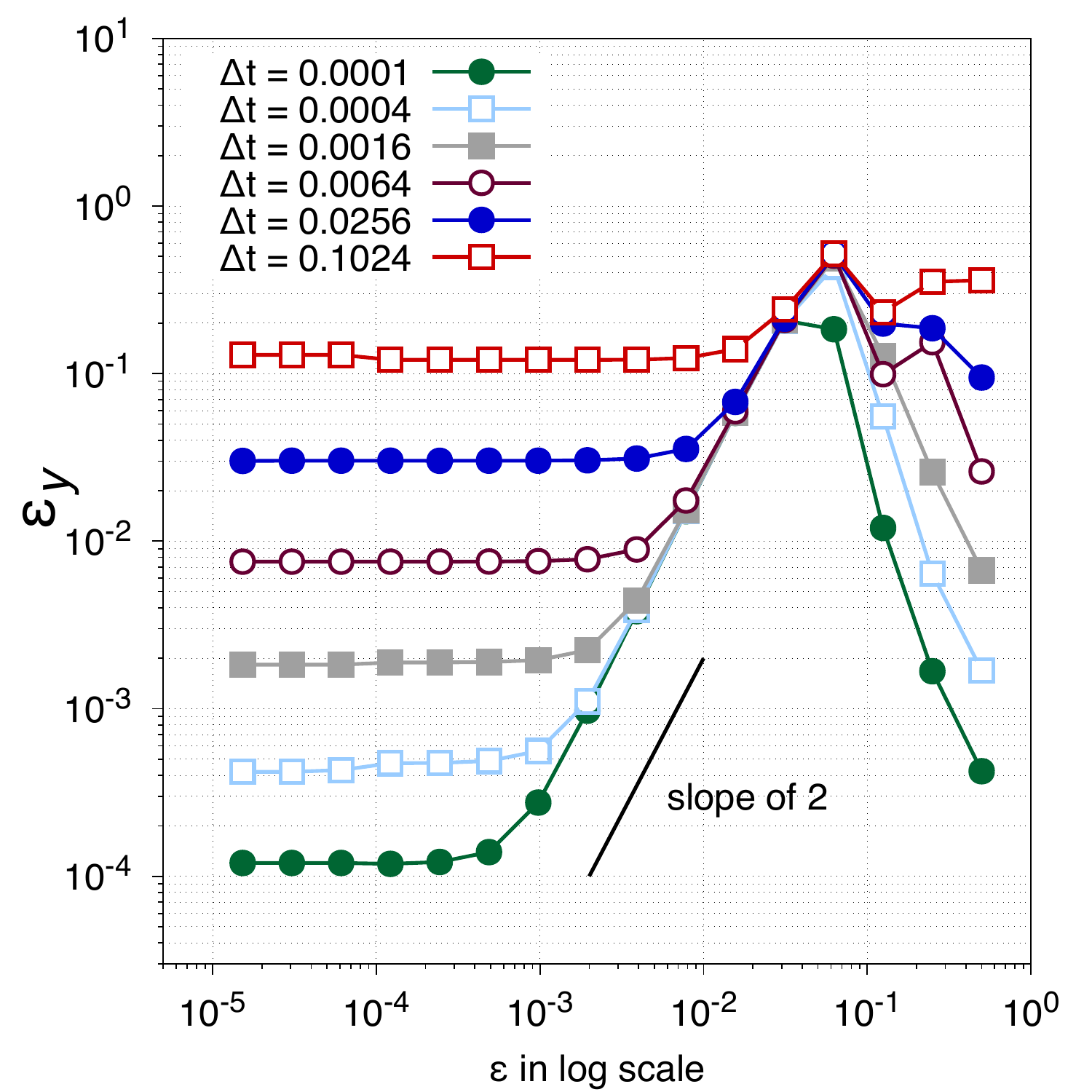}
\subcaption{}
\end{subfigure}
\begin{subfigure}[b]{0.45\textwidth}
\centering
 \includegraphics[width=\textwidth]{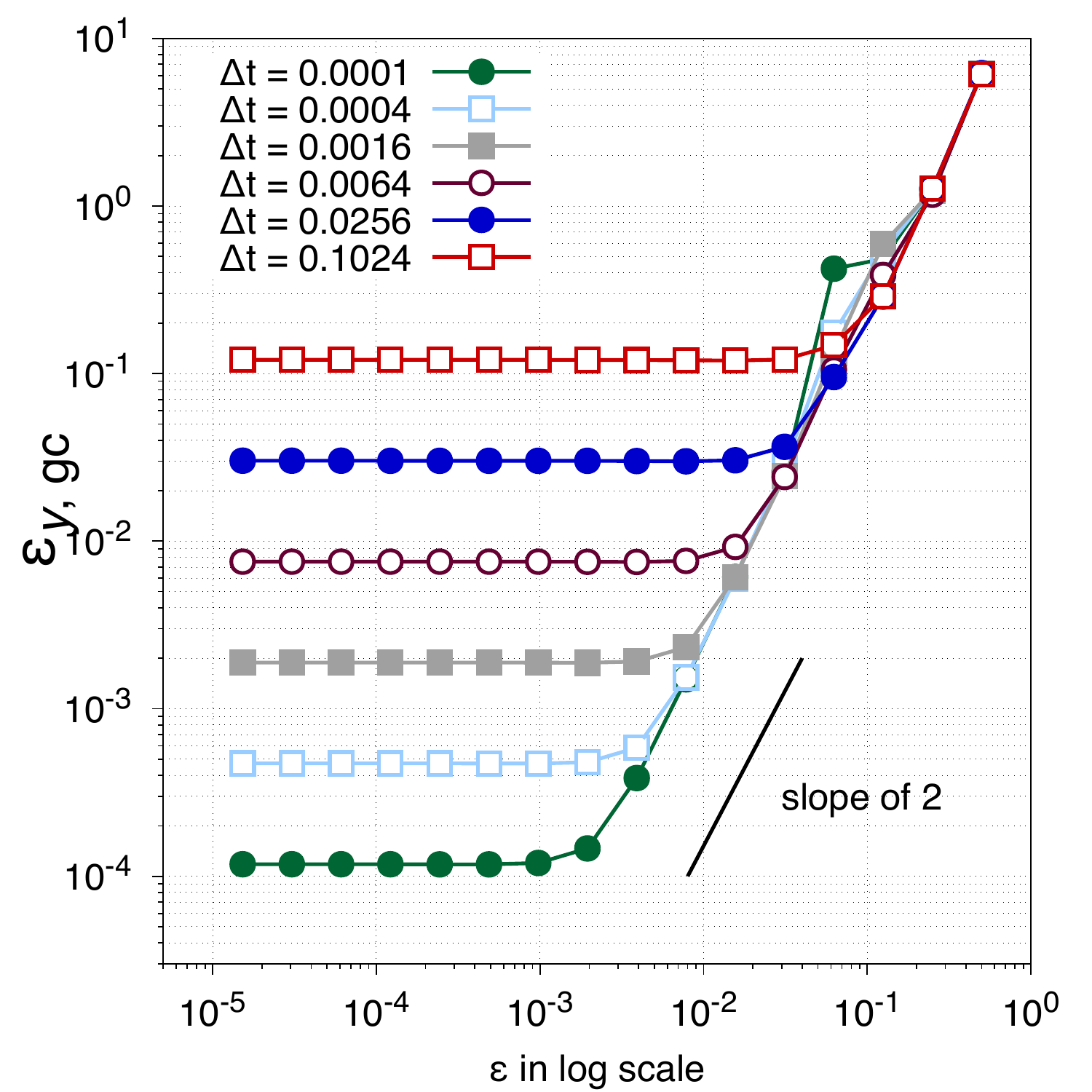}
\subcaption{}
\end{subfigure}
\begin{subfigure}[b]{0.45\textwidth}
\centering
 \includegraphics[width=\textwidth]{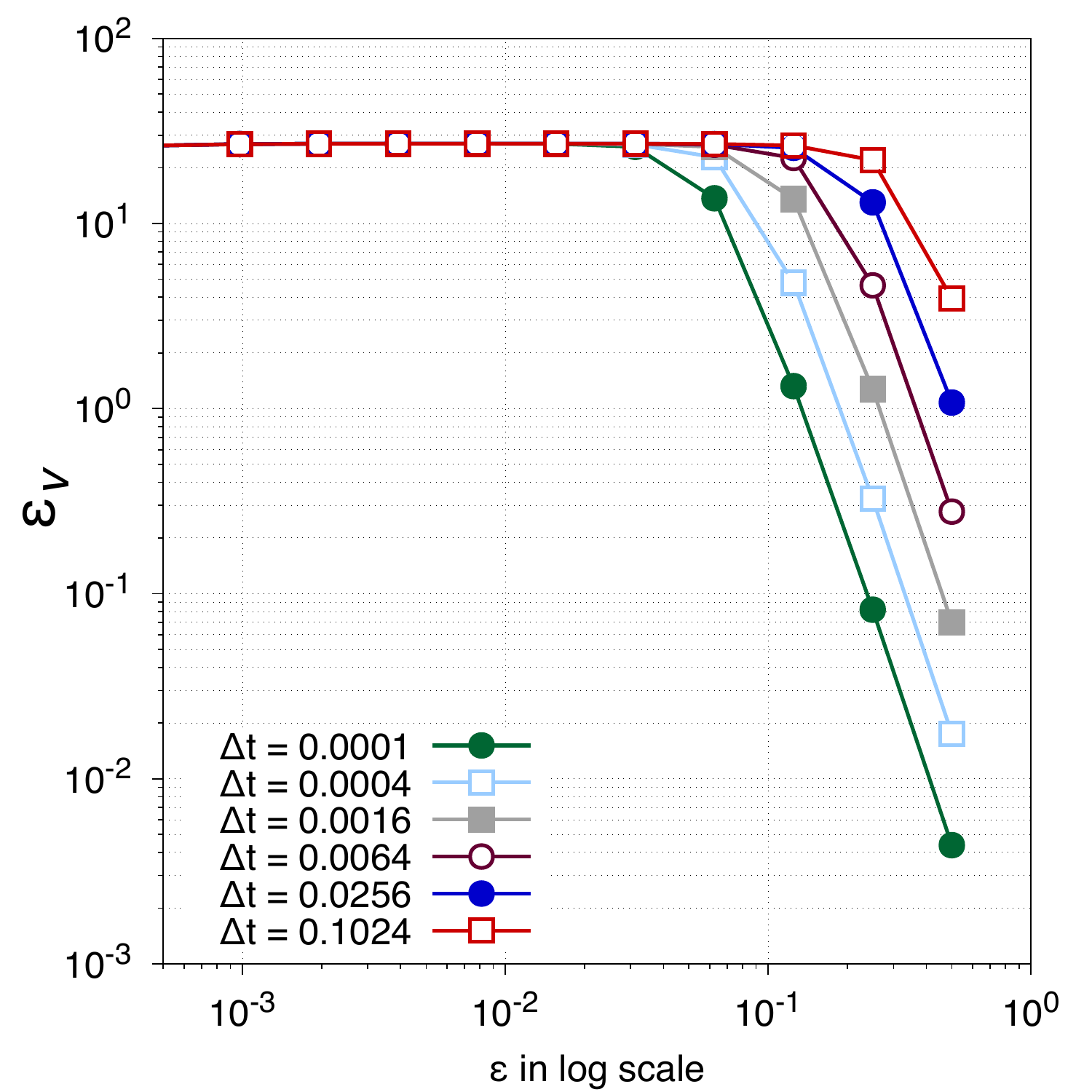}
\subcaption{}
\end{subfigure}
\begin{subfigure}[b]{0.45\textwidth}
\centering
 \includegraphics[width=\textwidth]{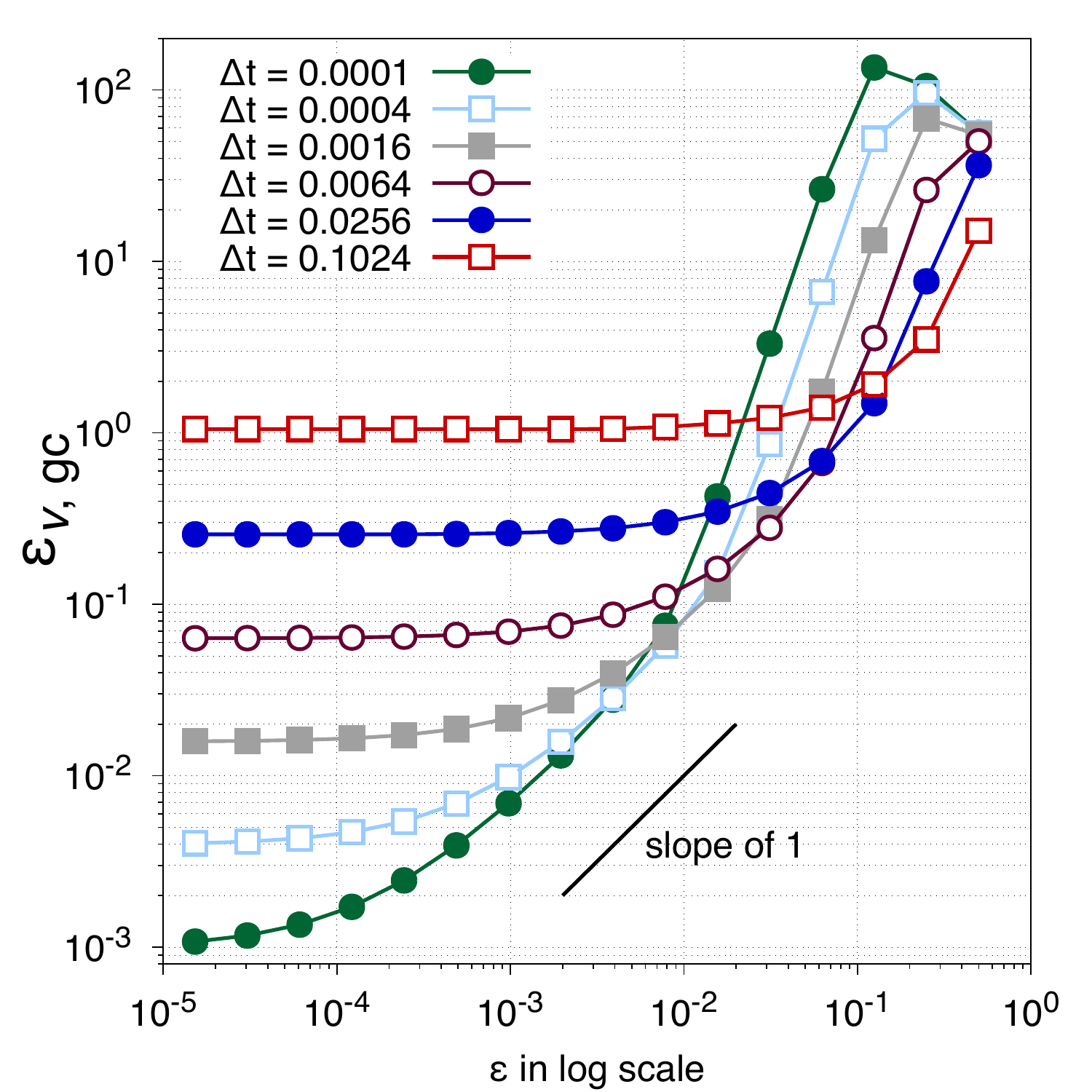}
\subcaption{}
\end{subfigure}
\caption{ {\bf First-order scheme \eqref{e:dNewton-1st_y}:} (A) Error
  between the approximation $\by_{\eps}^n$ and (reference)
  solution $\by_\eps(t_n)$ of \eqref{e:Newton} denoted by $\mathcal{E}_{\by}$.  (B)  Error
between the approximation $\by_{\eps}^n$ and the asymptotic
 (reference) solution $\bX(t_n,0,\bx^0-\eps(\bv^0)^\perp)$ of
  \eqref{e:Ngyro} denoted by $\mathcal{E}_{\by,\,\rm{gc}}$. (C) Error between the approximation $\bv_ \eps^n$ and the (reference)
  solution $\bv_\eps(t_n)$ of \eqref{e:Newton} denoted by $\mathcal{E}_{\bv}$.
    (D)  Error
  between the approximation $\eps^{-1}\bv_{\eps}^n$ and the asymptotic
  (reference) velocity $\bv_{\rm gc}(t_n)$ denoted by $\mathcal{E}_{\bv,\,\rm{gc}}$.}
\label{fig:01}
\end{figure}

Furthermore, we perform numerical experiments with the second-order
scheme \eqref{eq:lstab_stg} to illustrate Theorem~\ref{thm:error_estim_2nd}. The numerical results are shown in Fig.~\ref{fig:02}. We observe the same behavior of the numerical
error, but, of course, with a smaller error owing to the higher order of accuracy with respect to $\Dt$, though, again, there is an intermediate region where the dominant error term is $\cO(\eps^2)$. These numerical tests, with a smooth solution, underline the expected but clear advantage of the second-order scheme compared to the first-order approximation.  

\begin{figure}[h!]
\begin{subfigure}[b]{0.45\textwidth}
\centering
 \includegraphics[width=\textwidth]{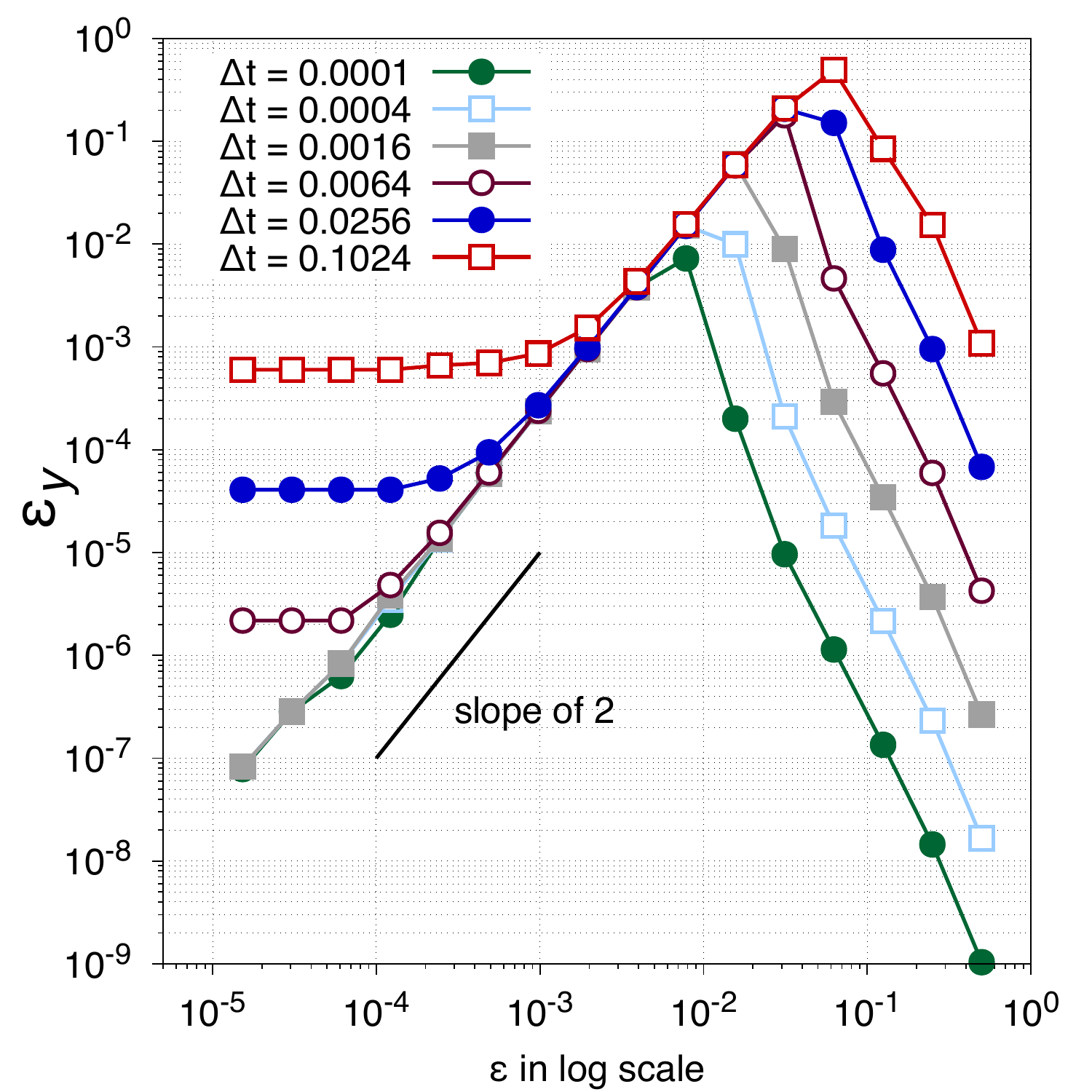}
\subcaption{}
\end{subfigure}
\begin{subfigure}[b]{0.45\textwidth}
\centering
 \includegraphics[width=\textwidth]{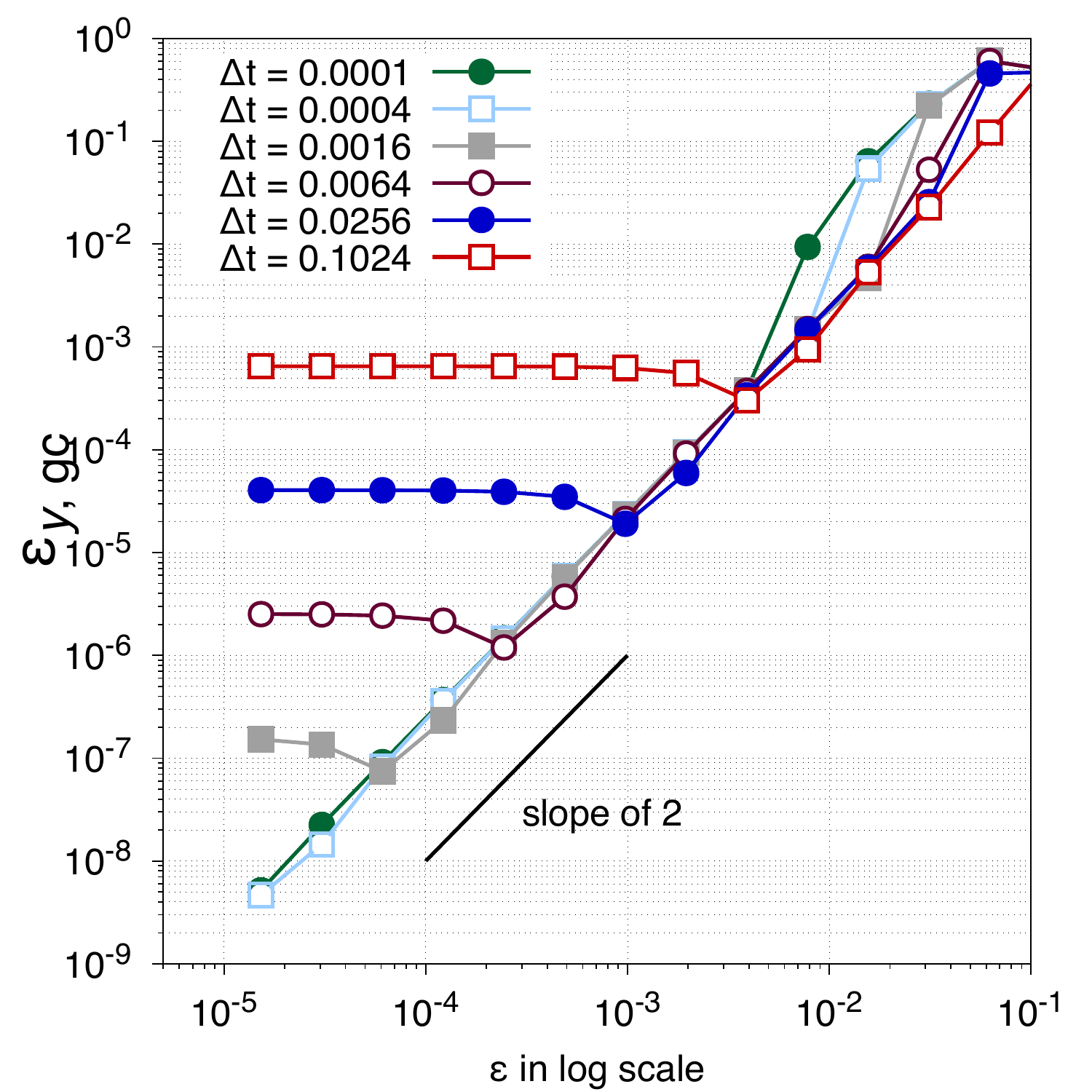}
\subcaption{}
\end{subfigure}
\begin{subfigure}[b]{0.45\textwidth}
\centering
 \includegraphics[width=\textwidth]{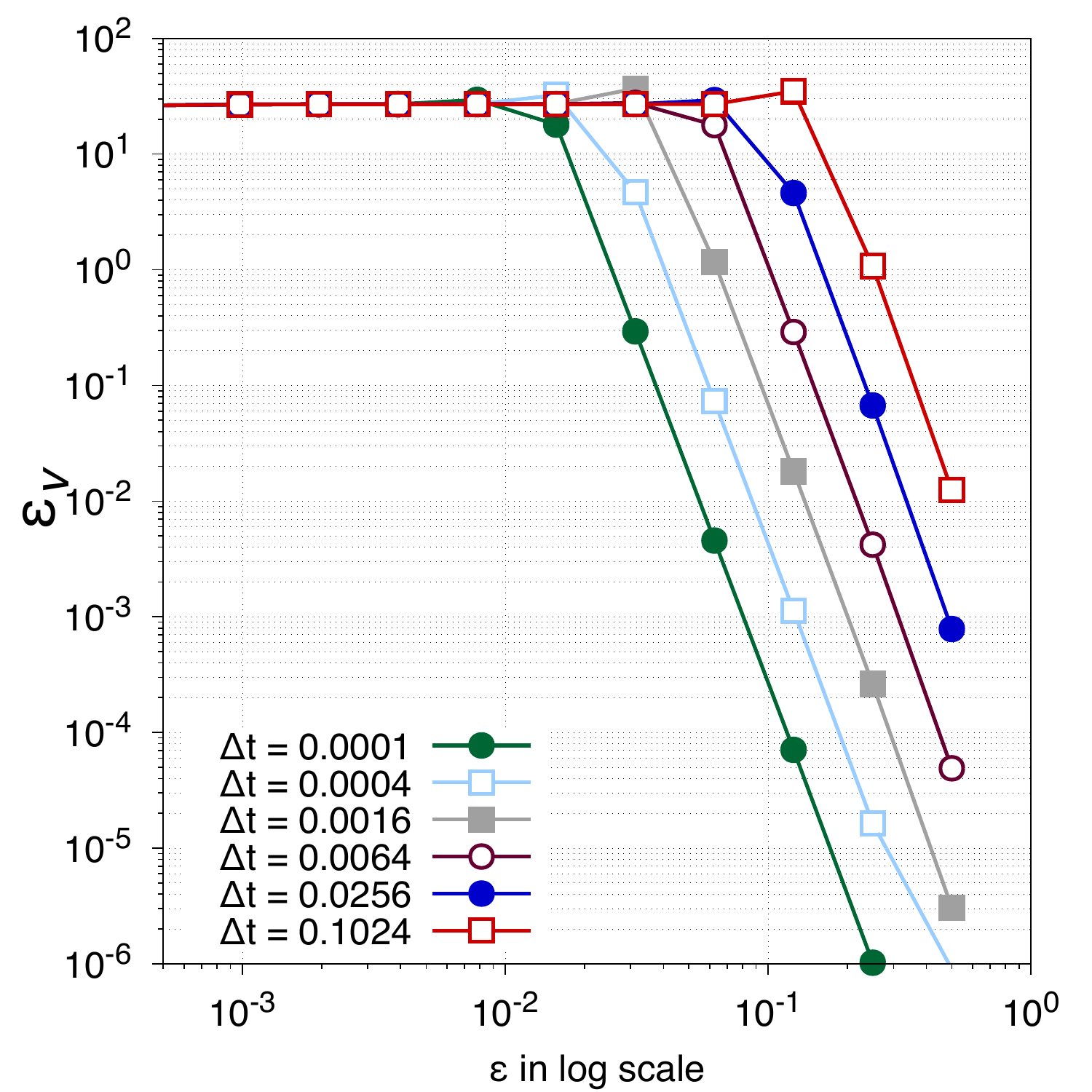}
\subcaption{}
\end{subfigure}
\begin{subfigure}[b]{0.45\textwidth}
\centering
 \includegraphics[width=\textwidth]{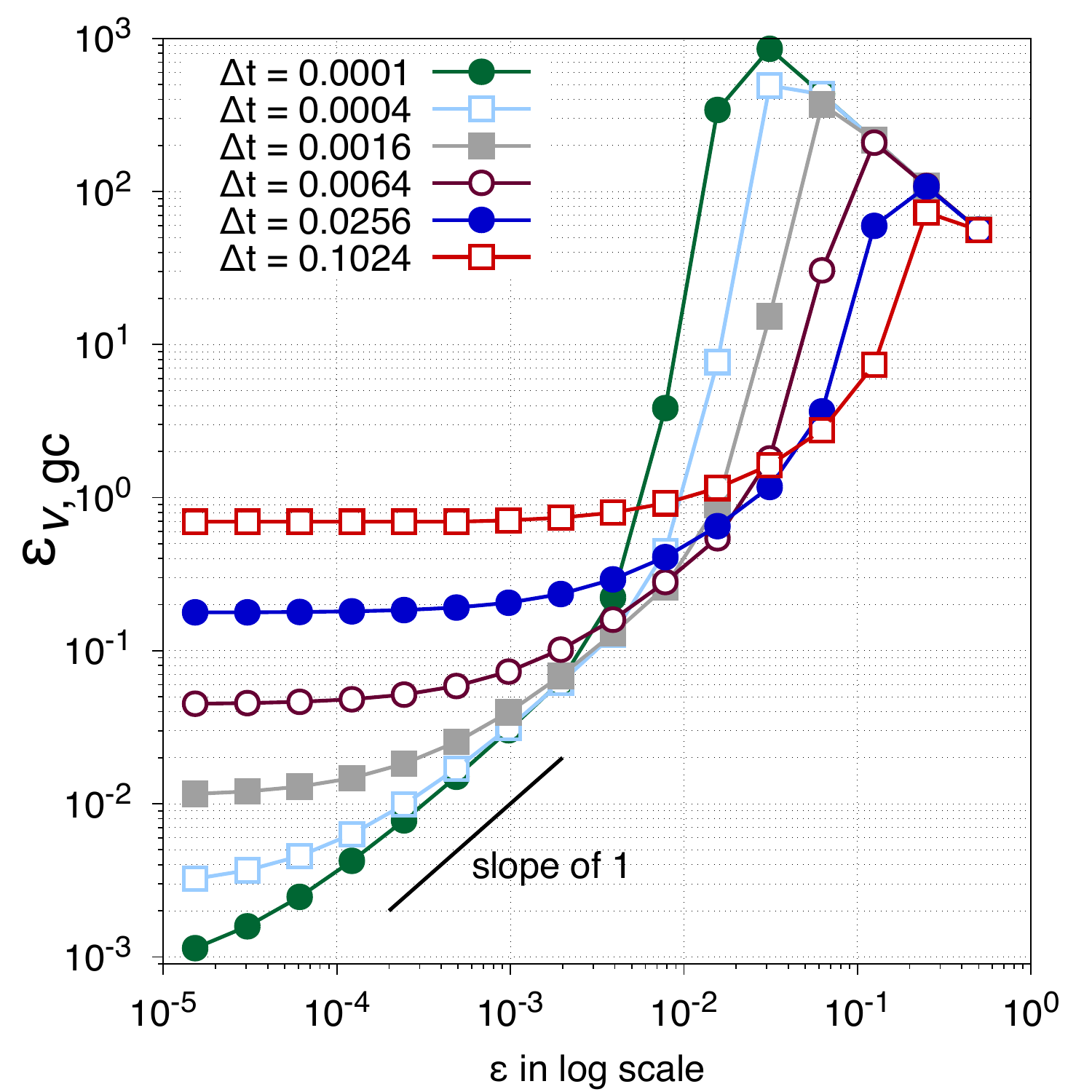}
\subcaption{}
\end{subfigure}
\caption{
{\bf Second-order scheme \eqref{eq:lstab_stg}:} (A) Error
  between the approximation $\by_{\eps}^n$ and (reference)
  solution $\by_\eps(t_n)$ of \eqref{e:Newton} denoted by $\mathcal{E}_{\by}$.  (B)  Error
between the approximation $\by_{\eps}^n$ and the asymptotic
 (reference) solution $\bX(t_n,0,\bx^0-\eps(\bv^0)^\perp)$ of
  \eqref{e:Ngyro} denoted by $\mathcal{E}_{\by,\,\rm{gc}}$. (C) Error between the approximation $\bv_ \eps^n$ and the (reference)
  solution $\bv_\eps(t_n)$ of \eqref{e:Newton} denoted by $\mathcal{E}_{\bv}$.
    (D)  Error
  between the approximation $\eps^{-1}\bv_{\eps}^n$ and the asymptotic
  (reference) velocity $\bv_{\rm gc}(t_n)$ denoted by $\mathcal{E}_{\bv,\,\rm{gc}}$..
   }
\label{fig:02}
\end{figure}

\section{Conclusion and perspectives}
In this paper, we have presented a complete convergence analysis of
particle-in-cell methods for the two-dimensional Vlasov equation, with a given electric field, submitted to an external magnetic field, which is homogeneous in space and time and very strong, of order $1/\eps$ with $\eps\ll 1$. In fact, we have estimated the error for semi-implicit first- and second-order IMEX schemes, and confirmed the stability, accuracy and convergence of these schemes, for any possible values of the time step and of the scaling parameter $\eps$. These theoretical results have been supported by numerical experiments.

An immediate extension is to investigate, in practice and analysis, the applicability of the presented framework for more complicated cases, \eg, for the three-dimensional system or with an inhomogeneous magnetic field. Another interesting extension could be to derive and analyze higher-order asymptotic models, to improve the error estimate in terms of $\eps$. This would be rather involved as higher order terms are coupled with the evolution of the energy; see \cite{FR17,filbet2018numerical} for instance.

\appendix
\section{Convergence analysis for oscillatory ODEs}\label{s:app}

As announced in the introduction, we conclude with abstract considerations on the numerical analysis of oscillatory ODEs. Though insufficient to prove the relevant results, these considerations provide enlightening insights supporting correct educated guesses on the final outcomes.

Let us discuss a system of the form 
\mm{\label{eq:abstract}
\begin{cases}
(\bfa_\eps+\eps^{r_t}\,\bG_t^\eps(\cdot,\bfa_\eps,\bfb_\eps))'(t)=\ds
\bF_\bfa^\eps(t,\bfa_\eps(t))+\eps^{r_x}\,\bG_x^\eps(t,\bfa_\eps(t),\bfb_\eps(t))\,,
\\[1em]
\ds
\bfb_\eps'(t)=-\frac{1}{\eps^2}\bJ\,\bfb_\eps(t)+\bF_\bfb^\eps(t,\bfa_\eps(t),\bfb_\eps(t))\,,
\end{cases}
}
(with $\bF_\bfa^\eps$, $\bF_\bfb^\eps$, $\bG_t^\eps$, $\bG_x^\epsilon$ uniformly smooth) and try to guess what may be expected on the numerical computation of the slow variable $\bfa_\eps$. Expanding the first equation of the system with the second suggests that, with such a goal in mind, a direct convergence analysis of a discretization of \eqref{eq:abstract} could be carried out by working with the vector $(\bfa_\eps,\eps^{\min(r_t-2,r_x)}\bfb_\eps)$, and, arguing recursively, that its $(m+1)$th derivative is bounded by a multiple of $\max(\eps^{-(2m-\min(r_t-2,r_x))_+},\eps^{-(2(m+1)-\min(r_t-2,r_x))})$. As a consequence, a direct convergence analysis of a scheme of order $m$ that would be unconditionally stable would result for the numerical approximation of $(\bfa_\eps,\eps^{\min(r_t-2,r_x)}\bfb_\eps)$, thus also of $\bfa_\eps$, into a bound on numerical error by a multiple of 
\[
(\Dt)^m\times\max\left(\frac{1}{\eps^{(2m-\min(r_t-2,r_x))_+}},\frac{1}{\eps^{(2(m+1)-\min(r_t-2,r_x))}}\right)\,.
\]
For concreteness note that when analyzing the computation of the guiding center, $r_t=3$ and $r_x=2$ so that the bound is $\Dt^m/\eps^{(2m+1)}$. The bound is somewhat optimal in the prediction of the computational error for $\eps^{\min(r_t-2,r_x)}\bfb_\eps$. With this respect note that even if by a particularly clever method, for instance through stroboscopic averaging, one is able to improve the computation of a fast variable at particularly well-chosen set of discrete times, this extra precision will be lost when recovering by interpolation from these discrete times an approximation of $\bfb_\eps$ on the whole continuous time interval.

System~\eqref{eq:abstract} suggests that the variable $\bfa_\eps$ is actually $\cO(\eps^{\min(r_t,r_x)})$-close as $\eps\to0$ to a solution $\bfa$ of the uncoupled non-stiff equation
\mm{\label{eq:limit-abstract}
\bfa_\eps'(t)=\ds
\bF_\bfa^\eps(t,\bfa(t))\,.
}
For a scheme of order $m$ consistent with the foregoing asymptotic and unconditionally stable this suggests a bound of the numerical error in the approximation of $\bfa_\eps$ by a multiple of
\[
\eps^{\min(r_t,r_x)}+\Dt^m\,.
\]
Note that to conclude to the latter bound it is sufficient to know that the $\bfa$-part of the solution of the discrete scheme for \eqref{eq:abstract} converge as $\eps\to0$ to a solution of a scheme of order $m$ for \eqref{eq:limit-abstract} with rate $\cO(\eps^{\min(r_t,r_x)}+\eps\,\Dt^{m\,\frac{\min(r_t,r_x)-1}{\min(r_t,r_x)}})$, leaving room for some depreciation of the continuous rate $\cO(\eps^{\min(r_t,r_x)})$.

This provides a final bound of the numerical error for the variable $\bfa_\eps$ by a multiple of 
\[
\min\left(\eps^{\min(r_t,r_x)}+\Dt^m,(\Dt)^m\times\max\left(\frac{1}{\eps^{(2m-\min(r_t-2,r_x))_+}},\frac{1}{\eps^{(2(m+1)-\min(r_t-2,r_x))}}\right)\right)\,.
\]
In the present paper, we have turned the foregoing formal discussion into rigorous convergence analysis for some of the schemes introduced in \cite{FR16} and a few extensions. 


\bibliographystyle{siam}
\bibliography{Ref} 

\end{document}